\documentclass[11pt]{amsart}

\usepackage{amsmath, amssymb, graphicx}

\newcommand{\bbN}{{\mathbb N}}

\newcommand{\bbC}{\mathbb C}

\newcommand{\cZ}{{\mathcal Z}}

\newcommand{\rs}{\restriction}

\DeclareMathOperator{\rd}{\chi_r}

\DeclareMathOperator{\id}{id} 

\newcommand{\cB}{B}

\newcommand{\e}{\varepsilon}
\newtheorem{theorem}{Theorem}[section] %%%%%%%%%%%%%%% TK changed
\newtheorem{corollary}[theorem]{Corollary}

\newtheorem{lemma}[theorem]{Lemma}
\newtheorem{prop}[theorem]{Proposition}

\renewcommand\labelenumi{(\arabic{enumi})}

\theoremstyle{definition}

\newtheorem{remark}[theorem]{Remark} %%%%%%%%%%%%%%% TK changed
\newtheorem{definition}[theorem]{Definition}

\newtheorem{problem}[theorem]{Problem}

\newcounter{my_enumerate_counter}
\newcommand{\pushcounter}{\setcounter{my_enumerate_counter}{\value{enumi}}}
\newcommand{\popcounter}{\setcounter{enumi}{\value{my_enumerate_counter}}}

\def\rs{\restriction}

\newcommand{\lbl}{\label}

\newcommand{\cP}{{\mathcal P}}

\newcommand{\bbZ}{\mathbb Z}

\DeclareMathOperator{\Ad}{Ad}

\newcommand{\F}{\mathcal F}
\newcommand{\G}{\mathcal G}

\DeclareMathOperator{\tr}{tr}

\title{Nonseparable UHF algebras I: Dixmier's problem}

\author{Ilijas Farah}

\address{Department of Mathematics and Statistics\\
York University\\
4700 Keele Street\\
North York, Ontario\\ Canada, M3J 1P3\\
and Matematicki Institut, Kneza Mihaila 34, Belgrade, Serbia}

\urladdr{http://www.math.yorku.ca/$\sim$ifarah}

\email{ifarah@mathstat.yorku.ca}

\author{Takeshi Katsura}

\address{Department of Mathematics\\
Faculty of Science and Technology\\
Keio University\\
3-14-1 Hiyoshi, Kouhoku-ku, Yokohama\\
JAPAN, 223-8522}

\email{katsura@math.keio.ac.jp}


\date{\today.}
\begin{document}
\begin{abstract}

There are three natural ways to define UHF (uniformly hyperfinite)
C*-algebras, and all three definitions are equivalent for separable
algebras. In 1967 Dixmier asked whether the three definitions remain
equivalent for not necessarily separable algebras. 
We give a complete answer to this question. 
More precisely, we show that in small cardinality 
two definitions remain equivalent, 
and give counterexamples in other cases. 
Our results do not use any additional
set-theoretic axioms beyond the usual axioms, namely ZFC.

\end{abstract}
\maketitle

\section{Introduction}

Let $A$ be a C*-algebra and let $\e$ be a positive number. 
For an element $x$ of $A$ and a subset
$\F$ of $A$, we write $x \in_\e \F$ if there exists $y \in \F$ such
that $\| x - y \| < \e$. For two subsets $\F, \G$ of $A$, we write
$\F \subseteq_\e \G$ if $x \in_\e \G$ for all $x \in \F$. For each $n
\in \bbN$, we denote by $M_n(\bbC)$ the unital C*-algebra of all
$n\times n$ matrices with complex entries. A C*-algebra which is
isomorphic to $M_n(\bbC)$ for some $n \in \bbN$ is called a
\emph{full matrix algebra}.

\begin{definition}
A C*-algebra $A$ is said to be
\begin{itemize}
\item \emph{uniformly hyperfinite}
(or \emph{UHF})
if $A$ is isomorphic to a tensor product
of full matrix algebras.
\item \emph{approximately matricial}
(or \emph{AM})
if it has a directed family of
full matrix subalgebras with dense union.
\item \emph{locally matricial}
(or \emph{LM})
if for any finite subset $\F$ of $A$ and any $\e > 0$,
there exists a full matrix subalgebra $M$ of $A$
with $\F \subseteq_\e M$,
\end{itemize}
\end{definition}

For a definition of tensor products,
see Definition~\ref{Def:tensor}.
The property LM was called \emph{matroid}
in \cite[Definition~1.1]{Dix:Some}.
A UHF algebra is unital by definition,
and it is easy to see that UHF implies AM
and that AM implies LM.
In \cite[Theorem~1.13]{Glimm:On},
Glimm shows that a unital separable LM algebra is UHF
(see also \cite[Remark~1.3 and Theorem~1.6]{Dix:Some}).
Thus for separable C*-algebras,
the three conditions UHF, unital AM and unital LM coincide.
Dixmier asked
whether these three conditions coincide
for general C*-algebras in \cite[Problem~8.1]{Dix:Some}.
We show that this is not the case.
To state our results precisely,
we need the following notion.

\begin{definition}\label{D.cd}
The \emph{character density} $\chi(A)$ of a C*-algebra $A$
is the smallest cardinality of a dense subset of $A$.
\end{definition}

Hence $A$ is separable if and only if
its character density $\chi(A)$ is
the first infinite cardinal $\aleph_0$.
Note that $\chi(A)$
is equal to the smallest cardinality
of an infinite generating subset of $A$.

The following are our main results
which completely answer \cite[Problem~8.1]{Dix:Some}.
Note that $\aleph_1$ is the smallest uncountable cardinal.

\begin{theorem} \label{Thm1}
\begin{enumerate}
\item
For a C*-algebra with character density at most $\aleph_1$,
AM and LM are equivalent.
\item
For every cardinal $\kappa > \aleph_1$,
there exists a unital LM algebra with character density $\kappa$
which is not AM.
\item
For every cardinal $\kappa \geq \aleph_1$, 
there exists a unital AM algebra 
with character density $\kappa$ which is not UHF. 
\end{enumerate}
\end{theorem}

\begin{proof}
\begin{enumerate}
\item Follows from Proposition~\ref{P.AM=LM+AF}
and Proposition~\ref{P.AF-aleph-1}.
\item
Follows from Proposition~\ref{Prop:B=limCAR}
and Proposition~\ref{P:notAF}.
\item
Follows from Proposition~\ref{P.B-X}.
\qedhere
\end{enumerate}
\end{proof}

In  (3), we can also control the representation density 
(defined in Definition~\ref{Def:rd}) of the example 
(Theorem~\ref{T.rd.1}). 
In particular, 
we distinguish between AM algebras and UHF algebras
faithfully represented on a separable Hilbert space. 

Results similar to   (1) and (2) hold 
for approximately finite-dimensional (AF) algebras. 

\begin{definition}
A C*-algebra $A$ is said to be
\begin{itemize}
\item \emph{approximately finite-dimensional}
(or \emph{AF})
if it has a directed family of
finite-dimensional subalgebras with dense union.
\item \emph{locally finite-dimensional}
(or \emph{LF})
if for any finite subset $\F$ of $A$ and any $\e > 0$,
there exists a finite-dimensional subalgebra $D$ of $A$
with $\F \subseteq_\e D$.
\end{itemize}
\end{definition}

It is easy to see that AF implies LF.
In \cite[Theorem~2.2]{Bratt:Inductive}
Bratteli proved that for a separable C*-algebra,
AF and LF are equivalent.
We get the following.

\begin{theorem}
\begin{enumerate}
\renewcommand\labelenumi{(\arabic{enumi})'}
\item For a C*-algebra with character density at most $\aleph_1$,
AF and LF are equivalent.
\item For every cardinal $\kappa > \aleph_1$,
there exists an LF algebra with character density $\kappa$
which is not AF.
\end{enumerate}
\end{theorem}

\begin{proof}
\begin{enumerate}
\renewcommand\labelenumi{(\arabic{enumi})'}
\item Follows from Proposition~\ref{P.AF-aleph-1}.
\item
Follows from Proposition~\ref{Prop:B=limCAR}
and Proposition~\ref{P:notAF}.\qedhere
\end{enumerate}
\end{proof}

A C*-algebra is AM (resp.\ AF) 
if and only if
it is obtained as a direct limit of
full matrix algebras 
(resp.\ finite-dimensional algebras) 
over a general directed set
(not necessarily a sequence).
On the other hand, 
it is not hard to see 
that a C*-algebra is LM (resp.\ LF) 
if and only if
it is obtained as a direct limit of
(separable) AM (resp.\ AF) algebras 
(Lemma~\ref{L.AF-reflection}). 
Hence the two theorems above imply the following. 

\begin{corollary}
The classes of AM algebras and AF algebras 
are not closed under taking direct limits. 
\end{corollary}

Some of the results of the present paper were announced in 
\cite{Ka:Non-separable}. 
By extending our methods the first author constructed an AM algebra 
that has faithful irreducible representations 
both on a separable Hilbert space and 
on a nonseparable Hilbert space~(\cite{Fa:Graphs}).
In the sequel to this paper \cite{FaKa:NonseparableII} 
we show that the classification problems
for UHF and AM algebras are significantly different.

\subsection*{Organization of the paper}
In \S\ref{Sec:Pre} we set up the toolbox used in the paper. In
\S\ref{Sec:LMnotUHF} we use the Jiang--Su algebra to distinguish LM
algebras from UHF algebras. 
$\sigma$-complete directed systems are used in \S\ref{Sec:AMnotUHF} 
to distinguish between AM and UHF algebras. 
The relation between AM and LM algebras 
as well as the one between AF and LF algebras are explained in
\S\ref{Sec:AM=LM} and \S\ref{Sec:AM neq LM}. 
In \S\ref{S.rd} we introduce 
the representation density, 
and using it 
distinguish between AM algebras and UHF algebras
faithfully represented on a given Hilbert space.

\section{Preliminary}\label{Sec:Pre}

In the present section we fix the terminology and prove some standard facts
from set theory, $\sigma$-complete directed systems 
and tensor products (respectively).

\subsection{Set theory}

By $X\amalg Y$ we denote the disjoint union of sets $X$ and $Y$.
If $f\colon X\to Y$ and $Z\subseteq X$ 
then we write $f[Z]=\{f(z) : z\in Z\}$ instead of
the notation $f(Z)$ commonly accepted outside of set theory.
Let us denote the cardinality of a set $X$ by $|X|$.
The countable infinite cardinal 
and the smallest uncountable cardinal
are denoted by $\aleph_0$ and $\aleph_1$, respectively.
The smallest uncountable ordinal is denoted by $\omega_1$.

\begin{lemma}\label{Lem:aleph2}
Let $X$ be a set.
For each $x \in X$,
choose a countable subset $Y_x \subseteq X$
with $x \in Y_x$.
If $|X| > \aleph_1$ then
one can find two elements $x,y \in X$
such that $x \notin Y_y$ and $y \notin Y_x$.
\end{lemma}

\begin{proof}
Take $Z \subseteq X$ with $|Z| = \aleph_1$.
Choose $x \in X \setminus \bigcup_{z \in Z} Y_z$
and $y\in Z \setminus Y_x$.
Then $x$ and $y$ are as required.
\end{proof}

\begin{remark}
The conclusion may be false if $|X| \leq \aleph_1$. 
To see this consider $X = \omega_1$ 
and $Y_x = \{y \in \omega_1 : y \leq x\}$ for $x \in \omega_1$.
\end{remark}

\begin{definition}
A directed set $\Lambda$ is said to be \emph{$\sigma$-complete}
if every countable directed $Z\subseteq \Lambda$
has the supremum $\sup Z \in \Lambda$.
\end{definition}

The ordered set $\omega_1$ is $\sigma$-complete. 
The following is another
$\sigma$-complete directed set considered in this paper.

\begin{definition}\label{Def.[X]}
For an infinite set $X$,  
we denote by $[X]^{\aleph_0}$ the set of all countable infinite subsets
of $X$, considered as a directed set with respect to the inclusion.
\end{definition}

\begin{definition}
Let $\Lambda$ be a $\sigma$-complete directed set. A subset $\Lambda_0$ of
$\Lambda$ is said to be \emph{closed} if for every countable directed
$Z\subseteq \Lambda_0$ we have $\sup Z \in \Lambda_0$, and \emph{cofinal} if
for every $\lambda \in \Lambda$ there exists $\lambda_0 \in \Lambda_0$ such
that $\lambda \preceq \lambda_0$.

A closed and cofinal subset is called a \emph{club}. 
\end{definition}

A club is an abbreviation of a \emph{closed and unbounded} set. 
The condition `\emph{unbounded}' 
(meaning `not having an upper bound') 
is equivalent to `cofinal' 
for totally ordered sets such as $\omega_1$, 
but is strictly weaker than `cofinal' 
for general directed sets.
A widely accepted custom among set theorists is calling 
closed and \emph{cofinal} subsets of $[X]^{\aleph_0}$ 
\emph{closed and unbounded} sets (or \emph{clubs}). 
Reluctantly, we continue this unfortunate
abuse of terminology in our paper. 
This can be justified by the fact that 
$\omega_1$ and $[X]^{\aleph_0}$ are 
the only $\sigma$-complete directed sets 
that we will consider from the next section on.

\begin{lemma}\label{Lem:phi=id}
Let $\Lambda$ be a $\sigma$-complete directed set. Let $\Lambda_0$ and
$\Lambda_0'$ be clubs of $\Lambda$ and $\phi\colon \Lambda_0 \to \Lambda_0'$ be
an order isomorphism. Then there exists a club $\Lambda_{00}$ of $\Lambda$ such
that $\Lambda_{00} \subseteq \Lambda_0 \cap \Lambda_0'$ and $\phi
\rs_{\Lambda_{00}} =\id$.
\end{lemma}

\begin{proof}
Set $\Lambda_{00} := \{\lambda \in \Lambda_0 \cap \Lambda_0' : \phi(\lambda) =
\lambda\}$. It is easy to see that $\Lambda_{00}$ is closed. We will see that
it is cofinal. Take $\lambda \in \Lambda$. Since $\Lambda_0$ is cofinal,
there exists $\lambda_1 \in \Lambda_0$ with $\lambda \preceq \lambda_1$. Since
$\Lambda_0'$ is cofinal, there exists $\lambda_1' \in \Lambda_0'$ with
$\lambda_1 \preceq \lambda_1'$ and $\phi(\lambda_1) \preceq \lambda_1'$.
Recursively, we can find $\lambda_n \in \Lambda_0$ and $\lambda_n' \in
\Lambda_0'$ for $n=1,2,\ldots$ such that
\[
\lambda_n \preceq \lambda_n', \quad
\phi(\lambda_n) \preceq \lambda_n', \quad
\lambda_n' \preceq \lambda_{n+1}, \quad
\phi^{-1}(\lambda_n') \preceq \lambda_{n+1}.
\]
Then
\[
\lambda_{00} := \sup \{\lambda_n\}_{n=1}^\infty
= \sup \{\lambda_n'\}_{n=1}^\infty \in \Lambda_0 \cap \Lambda_0'
\]
satisfies $\phi(\lambda_{00}) = \lambda_{00}$.
Thus we have found $\lambda_{00} \in \Lambda_{00}$
with $\lambda \preceq \lambda_{00}$.
\end{proof}

\begin{lemma} \label{Lem:project clubs}
Let $X$ and $Y$ be infinite sets. 
For a club $C$ in $[X \amalg Y]^{\aleph_0}$, 
there exists a club $C_0$ in $[X]^{\aleph_0}$ 
such that for every $\mu_0 \in C_0$ 
there exists $\mu \in C$ with $\mu_0 = \mu \cap X$. 
\end{lemma} 

\begin{proof} 
This is a well-known and very useful fact. 
We provide a proof for the reader's convenience.

Let $[X]^{<\aleph_0}$ denote the set of all finite subsets of $X$.
Since $C$ is cofinal, 
we can find an increasing map $f\colon [X]^{<\aleph_0}\to C$
satisfying $s \subseteq f(s)$ for all $s \in [X]^{<\aleph_0}$ 
by  induction on $|s|$.  
We define $g\colon [X]^{\aleph_0} \to [X \amalg Y]^{\aleph_0}$
by $g(\mu_0) := \bigcup_{s \subseteq \mu_0}f(s)$
for $\mu_0 \in [X]^{\aleph_0}$. 
For every $\mu_0 \in [X]^{\aleph_0}$, 
we have $\mu_0 \subseteq g(\mu_0)$ 
and $g(\mu_0) \in C$ 
because $f$ is increasing and $C$ is closed. 
We set 
\[
C_0 := \{ \mu_0 \in [X]^{\aleph_0} : \mu_0 = g(\mu_0) \cap X \}. 
\]
Then $C_0$ is closed 
because for a countable directed $Z\subseteq [X]^{\aleph_0}$, 
we have 
\[
\bigcup_{\mu_0 \in Z}g(\mu_0)=g\Big(\bigcup_{\mu_0 \in Z}\mu_0\Big). 
\]
It remains to show that $C_0$ is cofinal in $[X]^{\aleph_0}$. 
Take $\lambda_0 \in [X]^{\aleph_0}$ arbitrarily. 
We define $\lambda_1,\lambda_2, \ldots \in [X]^{\aleph_0}$ by 
$\lambda_{n+1} := g(\lambda_n) \cap X$ for $n=0,1,\ldots$. 
Then $\{\lambda_n\}_{n=0}^\infty$ is an increasing 
sequence in $[X]^{\aleph_0}$ 
and $\mu_0 := \bigcup_{n=0}^\infty \lambda_n$ is in $C_0$. 
Thus $C_0$ is cofinal. 
Therefore we get a club $C_0$ in $[X]^{\aleph_0}$ as required. 
\end{proof} 

We note that by a well-known result of Kueker for every  
club $C$ in $[X]^{\aleph_0}$
there exists $h\colon [X]^{<\aleph_0}\to X$ such that 
every $\mu\in [X]^{\aleph_0}$ closed under $h$ belongs to $C$.

\subsection{$\sigma$-complete directed families of subalgebras}

By a subalgebra of a C*-algebra we mean a C*-subalgebra,
and by a unital subalgebra of a unital C*-algebra we mean a
C*-subalgebra containing the unit of the original C*-algebra.
By a directed family $\{A_\lambda\}_{\lambda\in \Lambda}$ of subalgebras of a
C*-algebra $A$, we mean that $\Lambda$ is a directed set,
and $\lambda \preceq \mu$ if and only if $A_\lambda \subseteq A_\mu$.
Thus by definition $\Lambda \ni \lambda \mapsto A_\lambda$ is injective. 

\begin{definition}
A directed family $\{A_\lambda\}_{\lambda\in \Lambda}$
of subalgebras of a C*-algebra $A$
is said to be \emph{$\sigma$-complete}
if $\Lambda$ is $\sigma$-complete
and for every countable directed $Z\subseteq \Lambda$,
$A_{\sup Z}$ is the closure of the union of $\{A_\lambda\}_{\lambda\in Z}$.
\end{definition}

In other words, 
a directed family $\{A_\lambda\}_{\lambda\in \Lambda}$ 
is $\sigma$-complete 
if $\overline{\bigcup_{\lambda \in Z}A_\lambda}$ 
is in the family for 
every countable directed $Z\subseteq \Lambda$. 

\begin{lemma} \label{L:restriction to club}
Let $A$ be a C*-algebra, and let $\{A_\lambda\}_{\lambda\in \Lambda}$
be a $\sigma$-complete directed family of subalgebras of
$A$ with dense union. 
Then for a club $\Lambda_0 \subseteq \Lambda$,
the restriction $\{A_\lambda\}_{\lambda\in \Lambda_0}$
is also a $\sigma$-complete directed family with dense union.
\end{lemma}

\begin{proof}
The restriction $\{A_\lambda\}_{\lambda\in \Lambda_0}$
is $\sigma$-complete because $\Lambda_0$ is closed, 
and its union is dense because $\Lambda_0$ is cofinal. 
\end{proof}

\begin{lemma}\label{L:s-comp family}
Every C*-algebra $A$ has
a $\sigma$-complete directed family of
separable subalgebras with dense union.
\end{lemma}

\begin{proof}
We can take the family of
all separable subalgebras of $A$ ordered by the inclusion.
\end{proof}

\begin{lemma}\label{L:s-comp union}
Let $A$ be a C*-algebra, and let $\{A_\lambda\}_{\lambda\in \Lambda}$
be a $\sigma$-complete directed family of subalgebras of
$A$ with dense union. 
For every separable subalgebra $A_0$ of $A$ 
there exists $\lambda\in \Lambda$ such that $A_0 \subseteq A_\lambda$.
\end{lemma}

\begin{proof}
Let $\{a_1,a_2,\ldots\}$ be a dense sequence of $A_0$.
For each $n \in \bbN$,
one can inductively find $\lambda_n \in \Lambda$
such that $a_i \in_{1/n} A_{\lambda_n}$
for $i = 1, 2, \ldots, n$ and $\lambda_{n-1} \preceq \lambda_n$
because the family $\{A_\lambda\}_{\lambda\in \Lambda}$ is directed 
and its union is dense in $A$.
Then $\lambda := \sup \{\lambda_n:n\in\bbN \} \in \Lambda$
satisfies $A_0 \subseteq A_\lambda$.
\end{proof}

By the lemma above,
we can see that
the union of a $\sigma$-complete directed 
family is automatically closed.

\begin{prop}\label{Prop:find_club}
Let $A$ and $B$ be C*-algebras,
and $\{A_\lambda\}_{\lambda\in \Lambda}$
and $\{B_{\lambda'}\}_{\lambda'\in \Lambda'}$
be $\sigma$-complete directed families of
separable subalgebras of $A$ and $B$ with dense union.
Let $\Phi\colon A \to B$ be an isomorphism.
Then there exist clubs $\Lambda_0 \subseteq \Lambda$
and $\Lambda_0' \subseteq \Lambda'$
and an order isomorphism $\phi\colon \Lambda_0 \to \Lambda_0'$
such that $\Phi[A_\lambda] = B_{\phi(\lambda)}$
for all $\lambda \in \Lambda_0$.
If $\Lambda = \Lambda'$,
then one can take $\Lambda_0 = \Lambda_0'$
and $\phi = \id$.
\end{prop}

\begin{proof}
Let $\Lambda_0$ be the set of all $\lambda\in \Lambda$ such that
there exists $\lambda'\in \Lambda'$ with $\Phi[A_\lambda] =
B_{\lambda'}$. Similarly we define $\Lambda'_0 \subseteq \Lambda'$ as
the set of all $\lambda'\in \Lambda'$ such that there is $\lambda\in
\Lambda$ with $\Phi^{-1}[B_{\lambda'}]=A_\lambda$.
Then there exists an order isomorphism
$\phi\colon \Lambda_0 \to \Lambda_0'$ such that $\Phi[A_\lambda] =
B_{\phi(\lambda)}$ for all $\lambda \in \Lambda_0$. We are going to
show that $\Lambda_0 \subseteq \Lambda$ is a club. It is clear that
$\Lambda_0$ is closed. Take $\lambda \in \Lambda$. Since
$A_{\lambda}$ is separable, there exists $\lambda_1' \in \Lambda'$
such that $\Phi[A_\lambda] \subseteq B_{\lambda_1'}$ by
Lemma~\ref{L:s-comp union}. By the same reason, there exists
$\lambda_1 \in \Lambda$ such that $\Phi^{-1}[B_{\lambda_1'}]
\subseteq A_{\lambda_1}$. Then we have $A_\lambda \subseteq
A_{\lambda_1}$. In this way, we can find sequences
\begin{align*}
A_\lambda &\subseteq A_{\lambda_1} \subseteq
A_{\lambda_2} \subseteq A_{\lambda_3} \subseteq \cdots \\
&B_{\lambda_1'} \subseteq B_{\lambda_2'} \subseteq
B_{\lambda_3'} \subseteq \cdots
\end{align*}
such that $B_{\lambda_{n}'} \subseteq \Phi[A_{\lambda_n}]$
and $\Phi[A_{\lambda_n}] \subseteq B_{\lambda_{n+1}'}$
for $n=1,2,\ldots$.
Let $\lambda_0 \in \Lambda$
and $\lambda'_0 \in \Lambda'$ be the supremums
of $\{\lambda_n\}_{n=1}^\infty$ and $\{\lambda'_n\}_{n=1}^\infty$.
Then we have
$A_{\lambda_0} = \overline{\bigcup_{n=1}^\infty A_{\lambda_n}}$
and $B_{\lambda'_0} = \overline{\bigcup_{n=1}^\infty B_{\lambda'_n}}$.
Since $\Phi[A_{\lambda_0}] = B_{\lambda_{0}'}$,
we get $\lambda_0 \in \Lambda_0$.
This shows that $\Lambda_0$ is cofinal,
and hence it is a club.
Similarly $\Lambda'_0 \subseteq \Lambda'$ is a club.
This shows the former assertion.
The latter assertion follows from Lemma~\ref{Lem:phi=id}.
\end{proof}

\begin{lemma}\lbl{L.AF-reflection}
A C*-algebra $A$ is LF 
if and only if it has a $\sigma$-complete
directed family of 
separable AF algebras with dense union.
\end{lemma}

\begin{proof}
We only need to prove the direct implication. 
We see that $A$ has a $\sigma$-complete directed family 
of separable subalgebras $\{A_\lambda\}_{\lambda \in \Lambda}$ 
with dense union by Lemma~\ref{L:s-comp family}. 
Since by \cite[Theorem~2.2]{Bratt:Inductive}
every separable LF algebra is AF, 
it suffices to show that the set $\Lambda_0$ 
of all $\lambda \in \Lambda$ such that $A_\lambda$ is LF is a club. 
Clearly $\Lambda_0$ is closed. 
To show that $\Lambda_0$ is cofinal, 
it suffices to see that for any separable subalgebra $A_0$ of $A$, 
there exists a separable subalgebra $A_0'$ containing $A_0$
such that for any finite subset $\F$ of $A_0$ and any $\e > 0$, there
exists a finite-dimensional subalgebra $M$ of $A_0'$ with $\F \subseteq_\e M$.
This is easy to see.
\end{proof}

In the same way, one can show that a C*-algebra $A$ is LM 
if and only if it has a $\sigma$-complete directed family 
of separable AM subalgebras with dense union.

\begin{remark} \label{R.K-0.reflection}
Lemma~\ref{L.AF-reflection} is just a special case of the downward
L\"o\-wen\-heim--Skolem theorem for logic of metric structures
(\cite{BYBHU}, or \cite{FaHaSh:Model2} for a version suitable for study of C*-algebras
and II$_1$ factors). Similar arguments have been used by C*-algebraists to
reflect properties of nonseparable algebras to separable subalgebras
(see \cite[II.8.5]{Black:Operator}) such as for example simplicity or
the existence of the unique trace.
\end{remark}

\subsection{Tensor products}

In this subsection, 
we give a definition and some properties 
of tensor products of C*-algebras. 
We try to avoid using results on nuclear C*-algebras 
as much as possible. 
In fact, we use the nuclearity only in Proposition~\ref{P.UHF-criterion}
(and Lemma~\ref{Lem:comple}) 
which is used in the proof of Proposition~\ref{P.B-X}~(3).
We are interested in 
tensor products of possibly uncountably many unital C*-algebras, 
and for this purpose the maximal tensor products are 
easier to treat than the minimal ones. 
We remark that we mainly deal with nuclear C*-algebras 
for which there is no distinction 
between the minimal tensor products and the maximal ones. 

\begin{definition}\label{Def:mutuallycommutes}
A family $\{A_x\}_{x \in X}$ of subalgebras 
of a C*-algebra $A$ is said to \emph{mutually commute}
if for distinct $x,y \in X$, 
every element of $A_x$ commutes with every element of $A_y$. 
\end{definition}

\begin{definition}\label{Def:tensor}
For a family $\{A_x\}_{x \in X}$ of unital C*-algebras,
its (maximal) \emph{tensor product} $\bigotimes_{x \in X}A_x$
is the C*-algebra having (an isomorphic copy of)  $A_x$ as unital subalgebras 
for $x \in X$ satisfying the following two properties: 
\begin{enumerate}
\item the family $\{A_x\}_{x \in X}$ of subalgebras 
of $\bigotimes_{x \in X}A_x$ mutually commutes, 
and its union $\bigcup_{x \in X}A_x$
generates $\bigotimes_{x \in X}A_x$. 
\item for a unital C*-algebra $B$ 
and a family $\{\varphi_x\}_{x \in X}$ 
of unital $*$-ho\-mo\-mor\-phisms $\varphi_x \colon A_x \to B$ 
such that $\{\varphi_x[A_x]\}_{x \in X}$ is 
a mutually commuting family of unital subalgebras of $B$, 
there exists a unital $*$-ho\-mo\-mor\-phism 
$\varphi\colon \bigotimes_{x \in X}A_x \to B$ 
such that $\varphi \rs_{A_x} = \varphi_x$ for all $x \in X$. 
\end{enumerate}
When $A_x = A$ for all $x \in X$,
we simply write $\bigotimes_X A$ for $\bigotimes_{x \in X}A_x$.
\end{definition}

It is not difficult to see that 
the tensor product exists and is unique. 
A nice exposition of tensor products of
C*-algebras can be found e.g., in \cite{BrOz}.
The condition (2) is called the universal property 
of the tensor product. 
A nice exposition of universal C*-algebras can be found 
e.g., \cite[II.8.3]{Black:Operator}.

Let $A$ and $B$ be unital C*-algebras. 
Since we consider $A$ and $B$ as unital subalgebras of $A \otimes B$, 
each $a \in A$ and each $b \in B$ are considered as
elements of $A \otimes B$. 
Thus the product $ab \in A \otimes B$ makes sense 
whereas this element is usually denoted by $a \otimes b \in A \otimes B$.
Similarly, for a family $\{A_x\}_{x \in X}$ of unital C*-algebras 
and a finite family $\{a_x\}_{x \in Y}$ of elements 
with $a_x \in A_x$ for $x \in Y \subseteq X$, 
we denote by $\prod_{x \in Y}a_x \in \bigotimes_{x \in X}A_x$ 
the product of $\{a_x\}_{x \in Y}$. 
Note that this product does not depend on the order of multiplications 
because the family $\{a_x\}_{x \in Y}$ in $\bigotimes_{x \in X}A_x$ 
mutually commutes.

The referee pointed out that the version of the next lemma 
when $A$ is assumed to be nuclear and simple instead of LM is true
(cf.\ \cite[Corollary~9.4.6]{BrOz}). 
Since one can prove that LM algebras are nuclear and simple, 
this gives a proof of this lemma. 
We give an elementary proof for the reader's convenience.

\begin{lemma}\label{Lem:LM times B}
Let $A$ and $B$ be unital subalgebras of a unital C*-algebra $D$ 
commuting with each other. 
If $A$ is LM, then the natural map from $A\otimes B$ 
to the C*-subalgebra $C^*(A\cup B)$ of $D$ 
generated by $A\cup B \subseteq D$ 
is an isomorphism.
\end{lemma}
 
\begin{proof}
We first show the statement 
in the case that $A$ is a full matrix algebra $M_n(\bbC)$. 
Let $\{e_{i,j}\}_{i,j=1}^n$ be 
a matrix unit of $A \cong M_n(\bbC)$. 
Then every element of $A\otimes B$ can be written as 
$\sum_{i,j=1}^n e_{i,j} b_{i,j}$ 
for $b_{i,j}\in B$. 
In $C^*(A\cup B) \subseteq D$, 
we have 
\[
b_{i',j'}
= \sum_{k=1}^n e_{k,i'}
\Big( \sum_{i,j=1}^n e_{i,j} b_{i,j} \Big) e_{j',k}
\]
for $i',j' = 1,2,\ldots, n$. 
Hence if an element $\sum_{i,j=1}^n e_{i,j} b_{i,j} \in A\otimes B$ 
is sent to $0\in D$ by the natural map $A\otimes B \to D$, 
then $b_{i,j} = 0$ for all $i,j$ 
which implies $\sum_{i,j=1}^n e_{i,j} b_{i,j} = 0$ in $A\otimes B$. 
Thus when $A$ is a full matrix algebra, 
the natural map $A\otimes B \to C^*(A\cup B)$ is injective, 
and hence an isomorphism. 

Now suppose that $A$ is LM. 
Let $\pi\colon A\otimes B \to C^*(A\cup B)$ 
be the natural map. 
Take $x \in A\otimes B$. 
Take $\e>0$ arbitrarily. 
Then there exist $a_1,a_2,\ldots,a_n \in A$ 
and $b_1,b_2,\ldots,b_n \in B$ 
such that
\[
\Big\| x - \sum_{i=1}^n a_ib_i \Big\| < \e. 
\]
Since $A$ is LM, 
we may assume (by perturbing $a_i$'s slightly if necessarily) 
that $a_1,a_2,\ldots,a_n \in M$ 
for some unital full matrix subalgebra $M$ of $A$. 
Then by the first part of the proof, 
we have 
\[
\Big\| \sum_{i=1}^n a_ib_i \Big\|
= \Big\| \pi \Big( \sum_{i=1}^n a_ib_i \Big) \Big\|. 
\]
Hence we get 
\[
\big| \| x \| - \|\pi(x)\| \big| < 2 \e. 
\]
Since $\e$ was arbitrary, we have $\|x \| = \|\pi(x)\|$. 
This shows that 
the natural map $\pi\colon A\otimes B \to C^*(A\cup B)$ is injective, 
and hence an isomorphism. 
\end{proof}

We take advantage of   Lemma~\ref{Lem:LM times B}
and use  the notation $A \otimes B$ whenever it is justified by this lemma. 
Note that this lemma is false if we replace LM by LF. 
To see this, just consider $A=B=D=\bbC \oplus \bbC$. 
For a family $\{A_x\}_{x \in X}$ of unital C*-algebras, 
and unital subalgebras $D_x \subseteq A_x$, 
we sometimes denote by $\bigotimes_{x \in X} D_x$ the 
subalgebra of $\bigotimes_{x\in X} A_x$ generated by 
the mutually commuting family $\{D_x\}_{x \in X}$ 
of unital subalgebras of $\bigotimes_{x\in X} A_x$. 
In fact, this unital subalgebra is the image of 
the $*$-homomorphism from the tensor product $\bigotimes_{x \in X} D_x$ 
to $\bigotimes_{x\in X} A_x$, 
but no confusion should arise. 

We use the following well-known fact without mentioning.
We give its proof for the reader's convenience.

\begin{lemma}
Let $A$ and $B$ be unital C*-algebras,
and $A_0 \subseteq A$ and $B_0 \subseteq B$ be
unital subalgebras.
Then we have $(A_0\otimes B_0) \cap B = B_0$
in $A\otimes B$.
\end{lemma}

\begin{proof}
Take a state $\varphi$ of $A$.
Define a linear map $E\colon A\otimes B \to B$
by $E(ab)=\varphi(a)b$ for $a \in A$ and $b \in B$.
Since $E(b) = b$ for $b \in B$ and $E(A_0\otimes B_0) \subseteq B_0$,
we get $(A_0\otimes B_0) \cap B \subseteq B_0$.
The inverse inclusion is easy to see.
\end{proof}

For two families $\{A_x\}_{x \in X_1}$ and $\{A_x\}_{x \in X_2}$
of unital C*-algebras,
the tensor product $\bigotimes_{x \in X_1 \amalg X_2}A_x$ is
naturally isomorphic to
\[
\Big(\bigotimes_{x \in X_1}A_x\Big) \otimes
\Big(\bigotimes_{x \in X_2}A_x\Big).
\]
We identify these two tensor products. In particular, we can and will consider
$\bigotimes_{x \in Y}A_x$ as a unital subalgebra of $\bigotimes_{x \in X}A_x$
for a subset $Y$ of $X$.
We use the convention that $\bigotimes_{x \in Y}A_x = \bbC $
for $Y = \emptyset$.
We remark that the subalgebra $\bigotimes_{x \in Y}A_x$ 
coincides with $\bigotimes_{x \in X}A_x$ for a subset $Y$ of $X$ 
if and only if $A_x = \bbC$ for all $x \in X \setminus Y$. 

The following is easy to see.

\begin{lemma}\label{Lem:tensor}
Let $\{A_x\}_{x \in X}$ be an infinite family of unital C*-algebras, 
and set $A=\bigotimes_{x\in X} A_x$. 
Then $\{\bigotimes_{x \in \lambda}A_x\}_{\lambda \in [X]^{\aleph_0}}$ 
is a $\sigma$-complete directed system of subalgebras of
$A$ with dense union. \qed
\end{lemma}

\begin{lemma} \lbl{L.tensor.density}
If $A=\bigotimes_{x\in X} A_x$, $X$ is infinite,
and each $A_x$ is separable and not $\bbC$,
then the character density $\chi(A)$ of $A$ is equal to $|X|$.
\end{lemma}

\begin{proof}
Fix a countable dense $C_x\subseteq A_x$ for each $x$. 
Their union has cardinality $|X|$ and generates $A$. 
This shows $\chi(A) \leq |X|$. 
Take a subset $Z\subseteq A$ with cardinality less than $|X|$. 
For each $z \in Z$, there exists $\lambda_z \in [X]^{\aleph_0}$ 
with $z \in \bigotimes_{x \in \lambda_z}A_x$ by Lemma~\ref{Lem:tensor}. 
Since the set $\bigcup_{z \in Z}\lambda_z \subseteq X$ 
has cardinality less than $|X|$, 
we can find $x \in X$ outside of this set. 
Since $A_x$ is not $\bbC$, $Z$ is not dense in $A$. 
Hence $\chi(A) = |X|$.
\end{proof}

For a unitary $u$ of a unital C*-algebra $A$, 
an automorphism $\Ad u$ on $A$ is defined 
by $\Ad u (a) =uau^*$ for $a \in A$.
Let $\{A_x\}_{x \in X}$ be a family of unital C*-algebras. 
By the universality, a family $\{\alpha_x\}_{x \in X}$ of automorphisms
$\alpha_x$ on $A_x$ determines 
the automorphism $\alpha$ on $\bigotimes_{x \in X}A_x$ 
with $\alpha\rs_{A_x}=\alpha_x$ 
which we denote by $\bigotimes_{x \in X}\alpha_x$. 
For a subset $Y \subseteq X$ and a family $\{\alpha_x\}_{x \in Y}$ 
of automorphisms $\alpha_x$ on $A_x$, 
we denote by $\bigotimes_{x \in Y}\alpha_x$ 
the automorphism $\bigotimes_{x \in X}\alpha_x$
of $\bigotimes_{x\in X} A_x$ 
where $\alpha_x = \id_{A_x}$ for $x \in X\setminus Y$.
For unitaries $u_x \in A_x$ for $x \in Y$, 
we get an automorphism $\bigotimes_{x \in Y} \Ad u_x$ 
on $A = \bigotimes_{x\in X} A_x$. 
When $Y$ is finite, 
we get $\bigotimes_{x \in Y} \Ad u_x = \Ad u$ 
where $u = \prod_{x \in Y}u_x \in A$, 
but in general, $\bigotimes_{x \in Y} \Ad u_x$ 
is not in the form $\Ad u$ for a unitary $u$ of $A$.

\subsection{Relative commutants}

For a subset $A$ of a C*-algebra $B$,
we denote by $Z_B(A)$ the \emph{relative commutant}
(or \emph{centralizer}) of $A$ inside $B$;
\[
Z_B(A) := \{ b\in B : \text{$ab =ba$ for all $a \in A$}\}
\]
which is a subalgebra of $B$ if $A$ is closed under the $*$-operation 
(for example if $A$ is a subalgebra). 
We avoid the common notation $A'\cap B$ for
$Z_B(A)$ in order to increase the readability of certain formulas.
For a subset $A$ of a C*-algebra $B$, 
we denote by $C^*(A)$ the subalgebra generated by $A$. 
Note that $Z_B(C^*(A))=Z_B(A)$ for a subset $A$ closed under the
$*$-operation.
We also note that $Z_B(A_1 \cup A_2)=Z_B(A_1) \cap Z_B(A_2)$.

\begin{lemma}\label{Lem:A'=B}
Let $A$ and $D$ be unital C*-algebras.
If $A$ is LM, then $Z_{A\otimes D}(A)=D$.
\end{lemma}

\begin{proof}
It is clear from the definition of tensor products
that $Z_{A\otimes D}(A) \supset D$.
Take $x_0 \in Z_{A\otimes D}(A)$.
For any $\e>0$,
there exist elements $a_1,\ldots,a_n \in A$
and $d_1,\ldots,d_n \in D$ such that
$\|x_0 - \sum_{i=1}^n a_id_i\| < \e$.
Since $A$ is LM,
we may assume that $a_1,\ldots,a_n$ are
in a full matrix unital subalgebra $M$ of $A$.
Let $E\colon A\otimes D \to A\otimes D$
be a contractive linear map defined by
 $E(x) = \int_{U}uxu^* du$
for $x \in A\otimes D$
where $du$ is the normalized Haar measure
on the unitary group $U$ of $M$.
Since $x_0 \in Z_{A\otimes D}(A)$,
we have $E(x_0) = x_0$.
For $a \in M$ and $d \in D$,
we have $E(ad)=\tr(a)d$
where $\tr\colon M \to \bbC$ is
the normalized trace.
Hence we have
$\|x_0 - \sum_{i=1}^n \tr(a_i)b_i\| < \e$.
This means that $x_0 \in_\e D$.
Since $\e$ was arbitrary, $x_0 \in D$.
Thus we get $Z_{A\otimes D}(A) \subseteq D$,
and therefore $Z_{A\otimes D}(A) = D$.
We are done.
\end{proof}

By letting $D = \bbC$ in the lemma above, 
we see that the center $Z_{A}(A)$ 
of an LM algebra $A$ is $\bbC$. 
Thus one can write the conclusion of Lemma~\ref{Lem:A'=B} 
as $Z_{A\otimes D}(A) = Z_{A}(A) \otimes D$. 
The referee pointed out that $Z_{A\otimes D}(A)=Z_A(A)\otimes D$ holds 
for minimal tensor products by \cite[Theorem~1]{HaWa}. 
Since one can prove that LM algebras are nuclear and satisfy 
$Z_{A}(A) = \bbC$, 
this gives an indirect proof of Lemma~\ref{Lem:A'=B}.

To prove Proposition~\ref{P.B-X}~(3), 
we need some facts on nuclear C*-algebras 
(Lemma~\ref{Lem:comple} and Proposition~\ref{P.UHF-criterion}). 
When we apply Proposition~\ref{P.UHF-criterion} 
in the proof of Proposition~\ref{P.B-X}~(3), 
we use the fact that a UHF algebra is a tensor product 
of separable nuclear C*-algebras
because full matrix algebras are nuclear.
A nice exposition of nuclearity of
C*-algebras can be found e.g., in \cite{BrOz}.

\begin{lemma}\label{Lem:comple}
Let $A$ and $D$ be unital C*-algebras, and $A_0$ a unital subalgebra of $A$.
Suppose that $D$ is nuclear. 
Then $Z_{A\otimes D}(A_0)=Z_A(A_0)\otimes D$.
\end{lemma}

\begin{proof}
Clearly we have $Z_A(A_0)\otimes D\subseteq Z_{A\otimes D}(A_0)$.
Let
\[
F := \{c\in A\otimes D:
\text{$(\id \otimes \omega)(c)\in Z_A(A_0)$ for all $\omega \in D^*$}\}.
\]
For $a \in A \subseteq A\otimes D$ and $c \in A\otimes D$,
we have
\[
(\id \otimes \omega)(ac)=a(\id\otimes \omega)(c),\quad
(\id \otimes \omega)(ca)=(\id\otimes \omega)(c)a
\]
for all $\omega\in D^*$.
Hence we get $Z_{A\otimes D}(A_0)\subseteq F$.
We claim that $F=Z_A(A_0)\otimes D$.
This equality is usually referred to as the \emph{slice map property}
of the triple $(D,A,Z_A(A_0))$ (see \cite[Definition~12.4.3]{BrOz}).
Here we remark that 
the (maximal) tensor product considered in this paper 
coincides with the minimal one because $D$ is nuclear. 
By \cite[Theorem~12.4.4~(2)]{BrOz}
(see \cite[Definition~12.4.1]{BrOz} and
note nuclear$\Leftrightarrow$CPAP$\Rightarrow$SOAP),
the triple
$(D,A,Z_A(A_0))$ has the slice map property
because $D$ is nuclear.
Thus we have $Z_A(A_0)\otimes D=Z_{A\otimes D}(A_0)$.
\end{proof}

\begin{definition}
Let $A$ be a unital C*-algebra,
and $A_0$ a unital subalgebra of $A$.
We say that $A_0$ is \emph{complemented} in $A$
if $C^*(A_0\cup Z_A(A_0))=A$.
\end{definition}

In a tensor product $A = \bigotimes_{x \in X} A_x$
of unital C*-algebras $A_x$,
a subalgebra $A_Y = \bigotimes_{x \in Y} A_x$ is complemented
for every subset $Y$ of $X$.

\begin{prop} \label{P.UHF-criterion}
Let $A$ be a unital C*-algebra.
Suppose that there exists a unital C*-algebra $D$
such that $A \otimes D$ is
a tensor product of separable nuclear C*-algebras.
Then for a $\sigma$-complete directed system
$\{A_\lambda\}_{\lambda \in \Lambda}$
of separable subalgebras of $A$ with dense union,
there exists a club $\Lambda_0 \subseteq \Lambda$
such that for each $\lambda \in \Lambda_0$,
$A_\lambda$ is complemented in $A$.
\end{prop}

\begin{proof}
Fix a dense $X\subseteq A$ and a dense $Y\subseteq D$. 
Then $\{C^*(\mu)\}_{\mu \in [X \amalg Y]^{\aleph_0}}$ 
is a $\sigma$-complete directed family of 
separable subalgebras of $A\otimes D$ 
with dense union. 
Since $A \otimes D$ is a tensor product of separable C*-algebras,
$A \otimes D$ has a $\sigma$-complete directed system
of separable complemented subalgebras with dense union
by Lemma~\ref{Lem:tensor}.
Hence by Proposition~\ref{Prop:find_club},
there exists a club $C \subseteq [X \amalg Y]^{\aleph_0}$ such that
$C^*(\mu)$ is complemented in $A\otimes D$ for all
$\mu \in C$.
By Lemma~\ref{Lem:project clubs} 
there exists a club $C_0 \subseteq [X]^{\aleph_0}$
such that for every $\mu_0 \in C_0$ 
there exists $\mu \in C$ with $\mu_0 = \mu \cap X$. 
By Lemma~\ref{L:restriction to club}, 
$\{C^*(\mu_0)\}_{\mu_0 \in C_0}$ 
is a $\sigma$-complete directed family of separable subalgebras of $A$ 
with dense union. 
Hence by Proposition~\ref{Prop:find_club} applied with $\id\colon A\to A$, 
we get a club $\Lambda_0 \subseteq \Lambda$
such that for each $\lambda \in \Lambda_0$ 
there exists $\mu_0 \in C_0$ with $A_\lambda = C^*(\mu_0)$. 

It remains to prove that $A_\lambda$ is complemented in $A$ 
for every $\lambda \in \Lambda_0$. 
Take $\lambda \in \Lambda_0$. 
Then by the arguments above, 
there exists $\mu \in C$ 
such that $A_\lambda = C^*(\mu \cap X)$. 
Then $C^*(\mu)$ is complemented in $A\otimes D$, 
and we have $A_\lambda \subseteq C^*(\mu) \subseteq 
A_\lambda \otimes D$. 
Since $A \otimes D$ is a tensor product of nuclear C*-algebras,
$D$ is nuclear by \cite[Proposition~10.1.7]{BrOz}.
Hence we get $Z_A(A_\lambda)\otimes D=Z_{A\otimes D}(A_\lambda)$ 
by Lemma~\ref{Lem:comple}. 
Therefore we have 
\begin{align*}
A\otimes D
&=C^*\big(C^*(\mu) \cup Z_{A\otimes D}(C^*(\mu))\big)\\
&\subseteq C^*\big((A_\lambda \otimes D) \cup Z_{A\otimes D}(A_\lambda)\big) \\
&= C^*\big((A_\lambda \otimes D) \cup (Z_{A}(A_\lambda)\otimes D)\big) \\
&= C^*(A_\lambda \cup Z_{A}(A_\lambda)) \otimes D.
\end{align*}
This shows that $A_\lambda$ is complemented in $A$ 
and finishes the proof. 
\end{proof}

\section{LM but not UHF}\label{Sec:LMnotUHF}

Proposition~\ref{P.example.Z} gives examples of unital LM algebras that are not
UHF, answering part of \cite[Problem~8.1]{Dix:Some}. Recall that $\cZ$ is the
\emph{Jiang--Su algebra}. We shall need the following properties of $\cZ$
proved in \cite{JiSu:On}:
\begin{itemize}
\item $\cZ$ is a unital, separable C*-algebra which is not UHF.
\item $\bigotimes_{\aleph_0}\cZ \cong \cZ$.
\item $A\otimes \cZ \cong A$
for any infinite-dimensional separable UHF algebra $A$.
\end{itemize}

\begin{definition}
The UHF algebra $\bigotimes_{\aleph_0}M_2(\bbC)$
is called the \emph{CAR algebra}.
\end{definition}

\begin{prop}\label{P.example.Z}
For two sets $X$ and $Y$,
define $A_{X,Y} := \bigotimes_{X} M_2(\bbC)\otimes\bigotimes_{Y}\cZ$.
Suppose that $X$ is infinite.
Then we have the following.
\begin{enumerate}
\item $A_{X,Y}$ is a unital LM algebra with $\chi(A_{X,Y})=|X|+|Y|$.
\item $A_{X,Y}$ is UHF if and only if $|X|\geq |Y|$.
\item $A_{X,Y}\otimes D$ is UHF
for any UHF algebra $D$ with $\chi(D)\geq |Y|$.
\end{enumerate}
\end{prop}

\begin{proof}
Since $X$ is infinite, we can identify $A_{X,Y}$ with 
$\bigotimes_{X} A\otimes\bigotimes_{Y}\cZ$ 
where $A$ is the CAR algebra. 
For each $\lambda \in [X]^{\aleph_0}$ 
and $\lambda' \in [Y]^{\aleph_0}$ we set
\[
D_{\lambda,\lambda'}
:=\bigotimes_{\lambda} A\otimes\bigotimes_{\lambda'}\cZ
\subseteq A_{X,Y}
\]
Then $D_{\lambda,\lambda'}$ is the CAR algebra
for all $\lambda$ and $\lambda'$.
Since $\{D_{\lambda,\lambda'}\}$
is a $\sigma$-complete directed system with dense union,
we see that $A_{X,Y}$ is LM.
By Lemma~\ref{L.tensor.density},
we have $\chi(A_{X,Y})=|X|+|Y|$.
This shows (1).

By rearranging the factors,
we see that
$A_{X,Y}$ is UHF if $|X|\geq |Y|$
and that $A_{X,Y}\otimes D$ is UHF
for a UHF algebra $D$ with $\chi(D)\geq |Y|$.
It remains to show that $A_{X,Y}$ is UHF only if $|X|\geq |Y|$.
For the sake of obtaining a contradiction,
assume that $|X| < |Y|$ and $A_{X,Y}$ is UHF.
Let us denote by $A_x = M_2(\bbC)$ for $x \in X$ 
and $A_y = \cZ$ for $y \in Y$ 
the unital subalgebra of $A_{X,Y}$ 
so that $A_{X,Y} = \bigotimes_{x \in X} A_x\otimes\bigotimes_{y \in Y}A_y$.
Let $\Phi \colon A_{X,Y} \to \bigotimes_{z \in Z}M_z$ be an isomorphism 
where $Z$ is a set 
and $\{M_z\}_{z \in Z}$ is a family of full matrix algebras. 
For each $x \in X$, 
there exists a finite $F_x \subseteq Z$ 
such that $\Phi[A_x] \subseteq \bigotimes_{z \in F_x}M_z$. 
If we set $Z_1= \bigcup_{x \in X} F_x \subseteq Z$, 
then we get $|Z_1|=|X|$ and 
$\Phi\big[\bigotimes_{x \in X} A_x\big] 
\subseteq \bigotimes_{z \in Z_1}M_z$. 
Similarly, 
for each $z \in Z_1$ 
there exists a finite $G_z \subseteq Y$ 
such that $M_z \subseteq \Phi\big[\bigotimes_{x \in X} A_x 
\otimes\bigotimes_{y \in G_z}A_y\big]$. 
If we set $Y_1= \bigcup_{z \in Z_1} G_z \subseteq Y$, 
then we get $|Y_1| \leq |Z_1| = |X|$ and 
\[
\bigotimes_{z \in Z_1}M_z
\subseteq 
\Phi\Big[\bigotimes_{x \in X} A_x 
\otimes\bigotimes_{y \in Y_1}A_y\Big]. 
\]
Next for each $y \in Y_1$ 
there exists a countable $C_y \subseteq Z$ 
such that $\Phi[A_y] \subseteq \bigotimes_{z \in C_y}M_z$. 
If we set $Z_2= Z_1 \cup \bigcup_{y \in Y_1} C_y \subseteq Z$, 
then we get $Z_1 \subseteq Z_2$, $|Z_2|=|X|$ and 
\[
\Phi\Big[\bigotimes_{x \in X} A_x 
\otimes\bigotimes_{y \in Y_1}A_y\Big] 
\subseteq 
\bigotimes_{z \in Z_2}M_z. 
\]
Recursively, 
we can find increasing sequences 
$\{Y_k\}_{k=1}^n$ and $\{Z_k\}_{k=1}^n$ 
of subsets of $Y$ and $Z$, respectively, such that 
$|X \amalg Y_k| = |Z_k| = |X|$ and 
\[
\bigotimes_{z \in Z_k}M_z
\subseteq 
\Phi\Big[\bigotimes_{x \in X} A_x 
\otimes\bigotimes_{y \in Y_k}A_y\Big]
\subseteq 
\bigotimes_{z \in Z_{k+1}}M_z
\]
for $k=1,2,\ldots$. 
We set $Y' := \bigcup_{k=1}^\infty Y_k \subseteq Y$ 
and $Z' := \bigcup_{k=1}^\infty Z_k \subseteq Z$. 
Then we have $|X \amalg Y'| = |Z'| = |X|$ and 
\[
\Phi\Big[\bigotimes_{x \in X} A_x 
\otimes\bigotimes_{y \in Y'}A_y\Big] = 
\bigotimes_{z \in Z'}M_z. 
\]
Since $\bigotimes_{z \in Z'}M_z$ is UHF and hence LM, 
we have 
\begin{align*}
Z_{A_{X,Y}}\Big(\bigotimes_{x \in X} A_x 
\otimes\bigotimes_{y \in Y'}A_y\Big) &= 
\bigotimes_{y \in Y \setminus Y'}A_y,\\ 
Z_{\bigotimes_{z \in Z}M_z}\Big(\bigotimes_{z \in Z'}M_z \Big)
&= \bigotimes_{z \in Z \setminus Z'}M_z. 
\end{align*}
by Lemma~\ref{Lem:A'=B}. 
Thus we get 
$\Phi \big[\bigotimes_{y \in Y \setminus Y'}A_y\big]
=\bigotimes_{z \in Z \setminus Z'}M_z$.
Since $|Y'|\leq |X| < |Y|$, 
we see that $Y \setminus Y'$ is infinite. 
Hence $Z \setminus Z'$ is also infinite. 
By Proposition~\ref{Prop:find_club} 
and Lemma~\ref{Lem:tensor}, 
there exist $C \in [Y \setminus Y']^{\aleph_0}$ 
and $C' \in [Z \setminus Z']^{\aleph_0}$ 
such that 
$\Phi \big[\bigotimes_{y \in C}A_y\big]
=\bigotimes_{z \in C'}M_z$.
This is a contradiction 
because $\bigotimes_{y \in C}A_y \cong \cZ$ 
is not UHF. 
\end{proof}

We thank the referee who pointed out an error of a proof of 
Proposition~\ref{P.example.Z}~(3) in an earlier draft. 
By Proposition~\ref{P.example.Z}, the unital C*-algebra $A_{X,Y}$ is LM but not
UHF if $|X| < |Y|$. When $|X|=\aleph_0$ and $|Y|=\aleph_1$, we see that
$A_{X,Y}$ is AM by Theorem~\ref{Thm1}~(1). 
In the other case, we do not know whether $A_{X,Y}$ is AM or not.

\begin{problem} \label{Prob.AM?}
Let $X,Y$ be sets such that $\aleph_0 \leq |X| < |Y|$ and $\aleph_1<|Y|$.
Is $A_{X,Y} = \bigotimes_{X} M_2(\bbC)\otimes\bigotimes_{Y}\cZ$ AM?
\end{problem}

\section{AM but not UHF}\label{Sec:AMnotUHF}

In this section,
for each infinite set $X$
we define a unital AM-algebra $B_X$
with $\chi(B_X) = |X|$,
and show that
$B_X$, or even $B_X \otimes D$
for a unital C*-algebra $D$, is not UHF
when $|X| \geq \aleph_1$.

\begin{lemma}\label{Lem:M2}
A C*-algebra generated by two self-adjoint unitaries $v, w$ with $vw = -wv$ is
always isomorphic to $M_2(\bbC)$. 
\end{lemma}

\begin{proof}
A C*-algebra $A$ 
generated by two self-adjoint unitaries $v, w$ with $vw = -wv$ 
is spanned (as a vector space) by 4 elements $\{1,v,w,vw\}$, 
and it is noncommutative. 
Hence it is isomorphic to $M_2(\bbC)$ 
which is the unique noncommutative C*-algebra with dimension $\leq 4$. 
A concrete isomorphism from $A$ to $M_2(\bbC)$ is
given by sending $v$ and $w$ to the unitaries
\[
\left(\begin{matrix} 1 & 0 \\ 0 & -1 \end{matrix}\right)
\qquad \text{and} \qquad
\left(\begin{matrix} 0 & 1 \\ 1 & 0 \end{matrix}\right)
\]
in $M_2(\bbC)$.
\end{proof}

Let us take a set $X$.
For each $x \in X$,
let $A_x$ be a C*-algebra
generated by two self-adjoint unitaries $v_x, w_x$
with $v_x w_x = - w_x v_x$.
By Lemma~\ref{Lem:M2},
$A_x$ is isomorphic to $M_2(\bbC)$.
We define a UHF algebra $A_X$
by $A_X := \bigotimes_{x \in X} A_x \cong \bigotimes_X M_2(\bbC)$.
We define an automorphism $\alpha$ on $A_X$
by $\alpha := \bigotimes_{x \in X}\Ad v_x$.
Note that $\alpha^2 = \id$. 
Let $\{e_{i,j}\}_{i,j=1}^2$ be a system of matrix units of $M_2(\bbC)$, 
and define an embedding 
\[
\iota \colon 
A_X \ni a \mapsto ae_{1,1}+\alpha(a)e_{2,2} \in A_X\otimes M_2(\bbC).
\]
Let $u \in A_X\otimes M_2(\bbC)$ be a self-adjoint unitary 
defined by $u := e_{1,2} + e_{2,1}$. 
Set $B_X := C^*(\iota(A_X)\cup\{u\})$. 
We consider $A_X$ as a unital subalgebra of $B_X$ 
and omit $\iota$. 
Then we have $u a = \alpha(a) u$ for $a \in A_X$ 
and $B_X=\{au+a' : a,a'\in A_X\}$.

\begin{remark}\label{rem:crosspro}
The C*-algebra $B_X$ is nothing but 
the crossed product $A_X\rtimes_{\alpha} (\bbZ/2\bbZ)$.
\end{remark}

For $Y\subseteq X$,
we denote by $A_Y$ the subalgebra
$\bigotimes_{x \in Y} A_x \subseteq A_X \subseteq B_X$,
and define $B_Y := C^*(A_Y \cup \{u\}) \subseteq B_X$.
It is easy to see that $A_Y \subseteq A_X$
is globally invariant under $\alpha$,
and hence $B_Y=\{au+a' : a,a'\in A_Y\}$.

\begin{lemma} \lbl{Lem:A'}
If $Y$ is infinite
then $Z_{B_X}(A_Y)= A_{X\setminus Y}$.
\end{lemma}

\begin{proof}
Since $Z_{A_X}(A_Y) = A_{X\setminus Y}$ by Lemma~\ref{Lem:A'=B},
it suffices to show that $Z_{B_X}(A_Y) \subseteq A_X$.
Take $au+a'\in Z_{B_X}(A_Y)$ with $a,a' \in A_X$,
and we will show $a=0$.

For any $\e>0$
there is a finite $F\subseteq X$
such that $a \in_\e A_F$.
Since $Y$ is infinite,
pick $y\in Y\setminus F$.
The unitary $w_y \in A_Y$ satisfies $uw_y=-w_yu$.
Hence $w_y(au+a')=(au+a')w_y$ yields
$(w_ya+aw_y)u+(w_ya'-a'w_y)=0$. 
Since $b u + b' = 0$ for $b,b'$ in $A_X$ implies $b=b'=0$,
we have $w_ya = - aw_y$. 
Thus we get
\[
\|a\| = \|w_ya\| = \| w_y a + w_y a \| /2
= \| w_y a - a w_y \| /2.
\]
Since $a \in_\e A_F$ and $w_y$ commutes with $A_F$,
we have $\|a\|= \| w_y a - a w_y \| /2 < \e$.
Since $\e$ was arbitrary, $a=0$.
We are done.
\end{proof}

\begin{lemma} \lbl{L.2}
If $Y\subsetneq X$ and $Y$ is infinite, then $B_Y$ is not
complemented in~$B_X$.
\end{lemma}

\begin{proof}
Since $B_Y=C^*(A_Y\cup\{u\})$,
we have
\[
Z_{B_X}(B_Y) = Z_{B_X}(A_Y) \cap Z_{B_X}(\{u\})
= A_{X\setminus Y} \cap Z_{B_X}(\{u\})
\]
by Lemma~\ref{Lem:A'}.
Hence
\[
C^*(B_Y \cup Z_{B_X}(B_Y))
= \big\{au+a' :
a,a'\in A_Y \otimes (A_{X\setminus Y} \cap Z_{B_X}(\{u\}))\big\}
\]
which does not contain $w_y \in A_{X\setminus Y}$
for $y\in X\setminus Y$.
\end{proof}

\begin{prop}\label{P.B-X}
\begin{enumerate}
\item \lbl{T.1.1} If $X$ is infinite,
then $B_X$ is a unital AM algebra with $\chi(B_X)=|X|$.
\item \lbl{T.1.2} If $X$ is uncountable,
then $B_X$ is not UHF.
\item \lbl{T.1.3} If $X$ is uncountable,
then $B_X\otimes D$ 
is not UHF for any unital C*-algebra $D$.
\end{enumerate}
\end{prop}

\begin{proof}
\eqref{T.1.1} 
Suppose $X$ is infinite.

Let us set
\[
\Lambda=\{(F,y) : \text{$F \subseteq X$ finite,
and $y\in X \setminus F$}\}
\]
and define
\[
D_{(F,y)} = C^*\big(B_F \cup \{w_y\}\big) \subseteq B_X.
\]
for $(F,y) \in \Lambda$.
Then it is clear that
$\{D_{(F,y)}\}_{(F,y) \in \Lambda}$
is a directed family with dense union.
It remains to show that $D_{(F,y)}$ is a full matrix algebra
for $(F,y) \in \Lambda$.
Take $(F,y) \in \Lambda$ with $|F| = n \in \bbN$.
Then we have $A_F \cong M_{2^n}(\bbC)$,
and the restriction of $\alpha$ to $A_F$ coincides
with $\Ad v$ where
\[
v=\prod_{x \in F}v_x\in A_F.
\]
Then the two self-adjoint unitaries $uv$ and $w_y$ in $D_{(F,y)}$
satisfy $w_y(uv) = -(uv)w_y$ and commute with $A_F$.
By Lemma~\ref{Lem:M2},
the subalgebra of $D_{(F,y)}$ generated by $uv$ and $w_y$
is isomorphic to $M_2(\bbC)$, and commute with $A_F$.
Since $D_{(F,y)}$ is generated by $A_F$ and this subalgebra,
$D_{(F,y)}$ is isomorphic to $M_{2^{n+1}}(\bbC)$.
We are done.

\eqref{T.1.2} 
Suppose $X$ is uncountable. 
Then $\{B_Y\}_{Y \in [X]^{\aleph_0}}$ 
is a $\sigma$-complete directed family of separable
subalgebras of $B_X$ with dense union. 
By Lemma~\ref{L.2}, neither one of these subalgebras is complemented. 
By Lemma~\ref{Lem:tensor},
a UHF algebra has a $\sigma$-complete directed system 
of separable complemented subalgebras with dense union. 
Hence $B_X$ cannot be UHF by Lemma~\ref{Lem:project clubs}. 

\eqref{T.1.3} 
As in (2), 
$B_X$ has a $\sigma$-complete directed system 
of separable subalgebras with dense union 
neither one of which is complemented. 
By Proposition~\ref{P.UHF-criterion}, 
$B_X\otimes D$ cannot be UHF for any unital C*-algebra $D$
because every UHF algebra is a tensor product 
of separable nuclear C*-algebras.
\end{proof}

Note that an example of a unital LM algebra $A$
that is not UHF given in Proposition~\ref{P.example.Z}
has the property that $A \otimes D$ is UHF
for some UHF algebra $D$,
but the one given in Proposition~\ref{P.B-X}
does not have this property.

The following answers \cite[Problem~8.3]{Dix:Some} negatively 
although it was certainly known.

\begin{corollary} \label{Cor:DixProb8.3}
There is a proper subalgebra $A$ of the CAR algebra $B$ such that $A$ is also
CAR algebra and $Z_{B}(A)=\bbC 1$. 
In particular, $B\neq A\otimes Z_{B}(A)$.
\end{corollary}

\begin{proof} 
Use Proposition~\ref{P.B-X} with $X=\bbN$. 
Then $A_X$ is the CAR algebra. 
The C*-algebra $B_X$ is also the CAR algebra 
because it is a separable unital LM algebra 
obtained as a direct limit of algebras of the form $M_{2^n}(\bbC)$ 
by the proof of Proposition~\ref{P.B-X}.
By Lemma~\ref{Lem:A'}, we have $Z_{B_X}(A_X)=\bbC 1$.
\end{proof}

\section{AM = LM and AF = LF for character density $\leq \aleph_1$}\label{Sec:AM=LM}

We first show AM $=$ LM $+$ AF.
We use the following well-known result repeatedly.
Recall that a finite-dimensional C*-algebra $D$
is isomorphic to a direct sum of finitely many
full matrix algebras (e.g., \cite[Theorem~III.1.1]{Dav:C*}),
and the cardinality $|F|$ of a system $F$ of matrix units of $D$
as defined after \cite[Theorem~III.1.1]{Dav:C*}
coincides with the dimension of $D$.

\begin{lemma}[{\cite[Corollary~III.3.3]{Dav:C*}}]\label{L.Davidson}
Given $d\in\bbN$, there exists $\delta>0$
so that
if $D$ is a finite-dimensional subalgebra of a C*-algebra $A$
with a system $F$ of matrix units such that $|F| = d$
and $B$ is a subalgebra of $A$
such that $F\subseteq_\delta B$,
there exists a unitary $u$ in the unitization of $A$
satisfying $uD u^* \subseteq B$
and commuting with $D \cap B$.

Moreover, for a previously given $\e>0$ in addition to $d$,
there exists $\delta>0$ such that
one can choose $u$ as above so that $\|u-1\|<\e$. \qed
\end{lemma}

\begin{prop} \label{P.AM=LM+AF}
A C*-algebra is AM if and only if it is LM and AF.
\end{prop}

\begin{proof}
We only need to prove that
if a C*-algebra $A$ is LM and AF, then it is AM.
Take a directed family $\{D_\lambda\}_{\lambda\in\Lambda}$
of finite-dimensional subalgebras of $A$ with dense union.
To show that $A$ is AM,
it suffices to show that for any $\lambda\in\Lambda$
there exists a full matrix subalgebra $M$ containing $D_\lambda$
and contained in $D_{\lambda'}$ for some $\lambda' \succeq \lambda$.
Then the set of such full matrix subalgebras is directed
and has dense union.

Take $\lambda\in\Lambda$.
Let $F$ be a system of matrix units of $D_\lambda$.
Let $\delta>0$ be as in Lemma~\ref{L.Davidson} for $d=|F|\in\bbN$.
Since $A$ is LM,
it has a full matrix subalgebra $M_0$ such that $F\subseteq_\delta M_0$.
By Lemma~\ref{L.Davidson},
there exists a unitary $u$ in the unitization of $A$
satisfying $u D_\lambda u^* \subseteq M_0$.
Let $F'$ be a system of matrix units of $u^* M_0 u$.
Let $\delta'>0$ be as in Lemma~\ref{L.Davidson}
for $d=|F'|$.
Since $\{D_\lambda\}_{\lambda\in\Lambda}$ has dense union,
there exists $\lambda' \in \Lambda$ such that
$\lambda' \succeq \lambda$ and
$F' \subseteq_{\delta'} D_{\lambda'}$.
By Lemma~\ref{L.Davidson},
there exists a unitary $u'$ in the unitization of $A$
satisfying $u'(u^* M_0 u) {u'}^* \subseteq D_{\lambda'}$
and commuting with $(u^* M_0 u) \cap D_{\lambda'}$.
Since $D_\lambda \subseteq (u^* M_0 u) \cap D_{\lambda'}$,
the full matrix subalgebra
$M := u' (u^*M_0u){u'}^*$ of $A$
satisfies $D_\lambda \subseteq M \subseteq D_{\lambda'}$.
This completes the proof.
\end{proof}

By Proposition~\ref{P.AM=LM+AF},
the statement AM = LM is reduced to AF = LF
because LM implies LF.
Thus we only show that
AF = LF for character density at most $\aleph_1$
although the same argument as below works
for showing AM = LM directly
by just changing ``F'' to ``M'' and
``finite-dimensional'' to ``full matrix''
in all statements and proofs.

\begin{lemma} \label{L:well-known}
Let $A$ be a separable AF algebra.
For an increasing sequence $\{D_n\}_{n \in \bbN}$ of
finite-dimensional subalgebras of $A$
there exists an increasing sequence $\{D'_n\}_{n \in \bbN}$ of
finite-dimensional subalgebras with dense union
such that $\bigcup_{n \in \bbN}D_n \subseteq \bigcup_{n \in \bbN}D'_n$.
\end{lemma}

\begin{proof}
This is well-known to specialists, and can be shown in a similar way
to \cite[Theorem~III.3.5]{Dav:C*}. For the reader's convenience, we
give a proof.

Let $\{B_k\}_{k \in \bbN}$ be an increasing sequence of
finite-dimensional subalgebras of $A$ with dense union.
We construct inductively an increasing sequence
$\{k_n\}_{n \in \bbN}$ in $\bbN$ and
a sequence $\{u_n\}_{n \in \bbN}$
of unitaries in the unitization of $A$
with $\|u_n-1\|<2^{-n}$
such that for each $n\in \bbN$
the finite-dimensional algebra
\[
u_{n} \cdots u_2 u_1 D_n u_1^* u_2^* \cdots u_{n}^*
\]
is contained in $B_{k_n}$ and commutes with $u_{n+1}$.
We first construct $k_1$ and $u_1$.
Choose $k_1 \in \bbN$ such that $F \subseteq_{\delta} B_{k_1}$
where $F$ is a system of matrix units of $D_1$
and $\delta>0$ be as in
the latter statement of Lemma~\ref{L.Davidson}
for $d=|F|$ and $\e=2^{-1}$.
By Lemma~\ref{L.Davidson},
there exists a unitary $u_1$ in the unitization of $A$
satisfying $u_1D_1 u_1^* \subseteq B_{k_1}$
and $\|u_1-1\|<1/2$.
Suppose that $k_1,\ldots,k_{n-1} \in \bbN$ and
unitaries $u_1,\ldots,u_{n-1}$ were chosen.
Choose $k_n \in \bbN$ such that $k_n > k_{n-1}$ and
$F'\subseteq_{\delta'} B_{k_n}$
where $F'$ is a system of matrix units of
\[
u_{n-1}\cdots u_2 u_1 D_n u_1^* u_2^* \cdots u_{n-1}^*
\]
and $\delta'>0$ be as in
the latter statement of Lemma~\ref{L.Davidson}
for $d=|F'|$ and $\e=2^{-n}$.
Lemma~\ref{L.Davidson} gives us
a unitary $u_n$ in the unitization of $A$
satisfying
\[
u_{n} u_{n-1} \cdots u_2 u_1 D_n u_1^* u_2^* \cdots u_{n-1}^* u_{n}^*
\subseteq B_{k_n},
\]
commuting with
\[
u_{n-1}\cdots u_2 u_1 D_{n-1} u_1^* u_2^* \cdots u_{n-1}^*
\subseteq u_{n-1}\cdots u_2 u_1 D_{n} u_1^* u_2^* \cdots u_{n-1}^*
\cap B_{k_n}
\]
and satisfying $\|u_n-1\|<2^{-n}$.
Thus we get the desired sequences
$\{k_n\}_{n \in \bbN}$ and $\{u_n\}_{n \in \bbN}$.

Since $\|u_n-1\|<2^{-n}$ for all $n \in \bbN$
and $\sum_{n\in\bbN}2^{-n}=1<\infty$,
the sequence $\{u_{n} \cdots u_2 u_1\}_{n\in\bbN}$
converges to a unitary $u$ in the unitization of $A$.
Since
\begin{align*}
u_{n} \cdots u_2 u_1 D_n u_1^* u_2^* \cdots u_{n}^*
&=u_{n+1}
u_{n} \cdots u_2 u_1 D_n u_1^* u_2^* \cdots u_{n}^*u_{n+1}^*\\
&\subseteq
u_{n+1}
u_{n} \cdots u_2 u_1 D_{n+1} u_1^* u_2^* \cdots u_{n}^*u_{n+1}^*,
\end{align*}
$u_{n} \cdots u_2 u_1 D_n u_1^* u_2^* \cdots u_{n}^*$
commutes with $u_{n+2}$.
By repeating this argument,
one can see that
$u_{n} \cdots u_2 u_1 D_n u_1^* u_2^* \cdots u_{n}^*$
commutes with $u_{m}$ for all $m>n$.
Hence we get $uD_nu^*=u_{n} u_{n-1} \cdots u_2 u_1 D_n u_1^* u_2^* \cdots u_{n-1}^* u_{n}^*\subseteq B_{k_n}$.
We set $D'_n := u^* B_{k_n}u$ for $n \in \bbN$.
Then $\{D'_n\}_{n \in \bbN}$ is an increasing sequence of
finite-dimensional subalgebras with dense union
such that $\bigcup_{n \in \bbN}D_n \subseteq \bigcup_{n \in \bbN}D'_n$.
\end{proof}

In the next lemma,
for two families
$\Upsilon = \{D_\lambda\}_{\lambda \in \Lambda}$ and
$\Upsilon' = \{D'_\lambda\}_{\lambda \in \Lambda'}$
of subalgebras,
$\Upsilon \subseteq \Upsilon'$ means that
$\Lambda \subseteq \Lambda'$ and $D_\lambda=D'_\lambda$
for each $\lambda \in \Lambda$.

\begin{lemma} \label{L.LF-ctble}
Let $A$ be a separable AF algebra contained
in a separable AF algebra $A'$.
For a countable directed family $\Upsilon$
of finite-dimensional subalgebras of $A$ with dense union,
there exists a countable directed family
$\Upsilon'$ of finite-dimensional subalgebras of $A'$ with dense union
such that $\Upsilon \subseteq \Upsilon'$.
\end{lemma}

\begin{proof}
Let us write $\Upsilon = \{D_\lambda\}_{\lambda \in \Lambda}$.
Since $\Lambda$ is countable,
we can choose a subsequence $\{\lambda_n\}_{n \in \bbN}$
of $\Lambda$ such that
$\bigcup_{\lambda \in \Lambda}D_\lambda=\bigcup_{n \in \bbN}
D_{\lambda_n}$.
By Lemma~\ref{L:well-known},
there exists an increasing sequence $\{D'_n\}_{n \in \bbN}$
of finite-dimensional subalgebras of $A'$ with dense union
such that $\bigcup_{n \in \bbN}D_{\lambda_n}
\subseteq \bigcup_{n \in \bbN} D'_n$.
For each $\lambda \in \Lambda$,
there exists $n \in \bbN$ such that $D_\lambda \subseteq D'_n$
because $D_\lambda$ is finite-dimensional.
Let $\Lambda' := \Lambda \amalg \bbN$, ordered by
requiring that $\Lambda$ and $\bbN$ have their natural orderings and
$\lambda \preceq n$ if $D_\lambda \subseteq D'_n$.
Then the family $\Upsilon' := \{D'_\lambda\}_{\lambda \in \Lambda'}$
defined by $D'_\lambda := D_\lambda$ for $\lambda \in \Lambda$
satisfies the desired properties.
\end{proof}

\begin{lemma} \label{L.LF-aleph1}
Each LF algebra of character density at most $\aleph_1$
has a $\sigma$-complete directed family of
separable AF subalgebras with dense union indexed
by the ordinal $\omega_1$.
\end{lemma}

\begin{proof}
Let $A$ be an LF algebra with $\chi(A) \leq \aleph_1$.
Fix a dense subset $\{x_\gamma: \gamma \in \omega_1\}$ of $A$,
and define $A_\lambda := C^*(\{x_\gamma: \gamma<\lambda\})$
for each $\lambda \in \omega_1$.
Then $\{A_\lambda\}_{\lambda \in \omega_1}$ is
a $\sigma$-complete directed family of separable subalgebras of $A$.
By Lemma~\ref{L.AF-reflection},
$A$ also has a $\sigma$-complete direct family of
separable AF subalgebras with dense union.
By Proposition~\ref{Prop:find_club} applied with $\id\colon A\to A$,
there is a club $\Lambda\subseteq\omega_1$
such that $A_\lambda$ is AF for $\lambda\in \Lambda$.
As ordered sets, $\Lambda$ is isomorphic to $\omega_1$,
and $\{A_\lambda\}_{\lambda \in \Lambda}$ is
the desired family.
\end{proof}

\begin{prop} \lbl{P.AF-aleph-1}
Each LF algebra of character density at most
$\aleph_1$ is an AF algebra.
\end{prop}

\begin{proof}
Let $A$ be an LF algebra with $\chi(A) \leq \aleph_1$.
Let $\{A_\xi\}_{\xi \in \omega_1}$ be
a $\sigma$-complete directed family of
separable AF subalgebras of $A$
with dense union as in Lemma~\ref{L.LF-aleph1}.
Using transfinite recursion,
we are going to construct an increasing family of
countable directed families
$\Upsilon_\xi$ of finite-dimensional subalgebras
whose union is dense in $A_\xi$ for each $\xi \in \omega_1$.
For $\xi =0$,
choose an increasing sequence of
finite-dimensional subalgebras of $A_0$ with dense union,
and set it $\Upsilon_0$.
If $\Upsilon_\xi$ has been defined,
then $\Upsilon_{\xi+1}$ is defined using Lemma~\ref{L.LF-ctble}.
If $\eta$ is a limit ordinal and $\Upsilon_\xi$ has been defined
for all $\xi<\eta$,
let $\Upsilon_\eta=\bigcup_{\xi<\eta} \Upsilon_\xi$.
Since $A_\eta$ is the closure of the union of $\{A_\xi\}_{\xi<\eta}$,
$\Upsilon_\eta$ is as required.

Finally let $\Upsilon=\bigcup_{\xi \in \omega_1}\Upsilon_\xi$.
Then this is a directed family of finite-dimensional subalgebras of $A$
with dense union.
Thus $A$ is an AF algebra.
\end{proof}

The example of the following section easily shows that the version of
Lemma~\ref{L.LF-ctble} for nonseparable algebras is false.

\section{AM $\neq$ LM and AF $\neq$ LF for character density $> \aleph_1$}
\label{Sec:AM neq LM}

In this section, we construct an LM algebra which is not AF. This
C*-algebra shows the difference between the classes of AM and LM
algebras as well as between the classes of AF and LF algebras. To
show that a given C*-algebra is not AF, we use the following
criterion.

The converse direction in the following lemma was proved 
by George Elliott, 
following a remark  by Tamas Matrai, during the first author's talk at a set theory seminar in Toronto
in April 2009.  

\begin{lemma} \label{L.LM-criterion}
A C*-algebra  $A$ is AF if and only if 
there exists a map $\rho \colon A \to A$ 
such that $\|a - \rho(a) \| < 1$ for every $a \in A$ 
and $C^*(\{\rho(a)\}_{a\in F})$ is finite-dimensional
for every finite subset $F$ of $A$.
\end{lemma}

\begin{proof} 
Assume $A$ is AF and 
let $\{A_\lambda\}_{\lambda \in \Lambda}$ be 
a directed family of finite-dimensional subalgebras of $A$ 
with dense union.
For each $a\in A$
there exists $\lambda_a \in \Lambda$ 
such that there exists $\rho(a)\in A_{\lambda_a}$ 
with $\|a - \rho(a)\| < 1$.
For every finite subset $F$ of $A$
there exists $\lambda \in \Lambda$ such that 
$\lambda \succeq \lambda_a$ for all $a \in F$.
Then $C^*(\{\rho(a)\}_{a\in F}) \subseteq A_\lambda$ is finite-dimensional.

Now assume that $\rho \colon A \to A$ is 
as in the statement of the lemma. 
If $\Lambda$ is the family of all finite subsets of $A$ 
then $A_\lambda=C^*(\{\rho(a)\}_{a\in \lambda})$ 
form a directed family of finite-dimensional subalgebras of $A$. 
Fix $a\in A$ and $\e>0$. 
Let $\lambda=\{a/\e\}$. 
Then $\e \rho(a/\e)\in A_{\lambda}$ and $\|a- \e \rho(a/\e)\|<\e$. 
Since $a$ and $\e$ were arbitrary, we conclude $A$ is AF. 
\end{proof}

We also use the following lemma 
(for the case when $A$ is the CAR algebra) 
in the proof of Proposition~\ref{P:notAF}

\begin{lemma}\label{Lem:lin.indep}
Let $A$ be a unital LM subalgebra of a unital C*-algebra $B$. 
Take $a_1,a_2,\ldots,a_n \in A$ 
and $b_1,b_2,\ldots,b_n \in Z_B(A)$. 
If $(a_i)_{i=1}^n$ is linearly independent in $A$ 
and $\sum_{i=1}^na_ib_i=0$ in $B$, then
we have $b_i=0$ for all $i$.
\end{lemma}

\begin{proof}
Since $A$ is LM, 
the natural map from $A \otimes Z_B(A)$ to $B$ 
is injective by Lemma~\ref{Lem:LM times B}. 
It is well known that the inclusion map 
from the algebraic tensor product of $A$ and $Z_B(A)$ 
to the (maximal) tensor product $A \otimes Z_B(A)$ 
is injective (see \cite[II.9.1.3]{Black:Operator}). 
The conclusion follows from these lemmas. 
\end{proof}

\begin{definition}
We say that a pair $(v_1, v_2)$
of self-adjoint unitaries $v_1, v_2$
in a unital C*-algebra is \emph{generic} if the family
\begin{align*}
&\big((v_1v_2)^n, (v_1v_2)^nv_1\big)_{n \in \bbZ}\\
&=(1,v_1,v_2, v_1v_2, v_2v_1,
v_1v_2v_1, v_2v_1v_2, v_1v_2v_1v_2, v_2v_1v_2v_1,
v_1v_2v_1v_2v_1, \ldots )
\end{align*}
is linearly independent.
\end{definition}

In other words, $(v_1, v_2)$ is generic 
if and only if the map sending the natural generators of 
the group algebra $\bbC((\bbZ/2\bbZ) * (\bbZ/2\bbZ))$ 
to $v_1, v_2$ is injective.

\begin{lemma}\label{Lem:twopairs}
Let $v_1,v_2,w_1,w_2$ be the four self-adjoint unitaries
in the C*-algebra $C([0,1], M_2(\bbC))$ defined by
\begin{align*}
v_1(t)&=\left(\begin{matrix} 1 & 0 \\ 0 & -1 \end{matrix}\right), &
v_2(t)&=\left(\begin{matrix} \cos(\pi t) & \sin(\pi t) \\
\sin(\pi t) & -\cos(\pi t) \end{matrix}\right), \\
w_1(t)&=\left(\begin{matrix} 0 & 1 \\ 1 & 0 \end{matrix}\right), &
w_2(t)&=\left(\begin{matrix} -\sin(\pi t) & \cos(\pi t) \\
\cos(\pi t) & \sin(\pi t) \end{matrix}\right)
\end{align*}
for $t \in [0,1]$.
Then $v_1,v_2,w_1,w_2$ satisfy
$v_1 w_1 = - w_1 v_1$,
$v_2 w_2 = - w_2 v_2$
and the pair $(v_1,v_2)$ is generic.
\end{lemma}

\begin{proof}
It is routine to check the two equalities $v_1 w_1 = - w_1 v_1$ and
$v_2 w_2 = - w_2 v_2$. That the pair $(v_1,v_2)$ is generic comes
from the fact that $\{\cos(n\pi t)+\sqrt{-1}\sin(n\pi t)\}_{n \in
\bbZ}$ is linearly independent in $C([0,1])$. We leave the details to
the readers.
\end{proof}

Let $X$ be an infinite set, 
and $[X]^2$ be the set of all subsets of $X$ with cardinality $2$. 
For $\xi=\{x,y\}\in [X]^2$ let $A_\xi$ be the CAR algebra. 
We fix four self-adjoint unitaries 
$v_{x,y},v_{y,x},w_{x,y},w_{y,x}$ in $A_\xi$ 
such that $v_{x,y} w_{x,y} = - w_{x,y} v_{x,y}$, 
$v_{y,x} w_{y,x} = - w_{y,x} v_{y,x}$ 
and the pair $(v_{x,y}, v_{y,x})$ is generic. 
Such unitaries exist by Lemma~\ref{Lem:twopairs} 
because there exists a unital embedding 
from $C([0,1], M_2(\bbC))$ to the CAR algebra.

We define a UHF algebra $A_{[X]^2}$
by $A_{[X]^2}=\bigotimes_{\xi \in [X]^{2}} A_\xi \cong
\bigotimes_{[X]^{2}\times \aleph_0} M_2(\bbC)$.
For a subset $Y$ of $X$, we set
$A_{[Y]^2} = \bigotimes_{\xi \in [Y]^2} A_\xi \subseteq A_{[X]^2}$.

\begin{definition}\label{Def:group}
For a set $X$,
we denote by $G_X$ the abelian group consisting of
all finite subsets of $X$
where the operation is the symmetric difference $\Delta$.
\end{definition}

We often identify an element $x$ of $X$ with a subset $\{x\}$ of $X$.
Thus the group $G_X$ is generated by the family $\{x\}_{x\in X}$ of mutually commuting involutions.
Hence $G_X$ is isomorphic to the group $\bigoplus_{X}(\bbZ/2\bbZ)$
of the direct sum of $|X|$ copies of $\bbZ/2\bbZ$.

For $g \in G_X$ 
we define an automorphism $\alpha_g$ on $A_{[X]^2}$ by
\[
\alpha_g = \bigotimes_{\text{$x \in g$ and $y \notin g$}} \Ad v_{x,y}.
\]
If $x\notin g$ define unitaries $V_{g;x}$ and  $V_{x;g}$ in $A_{[X]^2}$ via
\[
V_{g;x}=\prod_{y\in g} v_{y,x} v_{x,y}\qquad\text{and} 
\qquad V_{x;g}=\prod_{y\in g} v_{x,y} v_{y,x}.
\]

\begin{lemma}\label{L.Vgx}
If $x\notin g$ then 
$\alpha_g \circ \alpha_x = \Ad (V_{g;x}) \circ \alpha_{g\cup \{x\}}$ and 
$\alpha_{g\cup \{x\}} \circ \alpha_x = \Ad (V_{x;g}) \circ \alpha_g$.
\end{lemma}

\begin{proof} 
Note that $v_{x,y}$ and $v_{z,t}$ commute unless $z=y$ and $x=t$. 
Using $x\notin g$ we have 
\begin{align*}
\alpha_g \circ \alpha_x & =
\Big(\bigotimes_{y\in g\text{ and } z\notin g} \Ad v_{y,z}\Big)
\circ 
\Big(\bigotimes_{z\neq x} \Ad v_{x,z}\Big)\\
&=\Big(\bigotimes_{y\in g} \Ad v_{y,x} 
\circ \!\!\!\! \bigotimes_{y\in g\text{ and } z\notin g\cup \{x\}} 
\!\!\!\!\!\!\!
\Ad v_{y,z}\Big)
\circ 
\Big(\bigotimes_{z\in g} \Ad v_{x,z}
\circ 
\!\bigotimes_{z\notin g\cup \{x\}}\!\!\Ad  v_{x,z}\Big)\\
&=\Ad (V_{g;x})\circ\alpha_{g\cup \{x\}}
\end{align*}
This proves the first equality. 
Since $\alpha_x$ is an involution and $V_{x;g}=V_{g;x}^*$, 
the first equality implies the second equality. 
\end{proof} 

Let us choose a faithful representation 
$A_{[X]^2} \subseteq B(H)$ on some Hilbert space $H$ 
(see Section~\ref{S.rd} for one construction of such a representation).
Let $\ell^2(G_X,H)$ be the Hilbert space consisting 
of functions $\xi: G_X \to H$ with 
$\sum_{g \in G_X}\|\xi(g)\|^2 < \infty$. 
We embed $A_{[X]^2}$ into $B(\ell^2(G_X,H))$ by 
\[
(a \xi)(g)=\alpha_g(a)\xi(g) \in H
\]
for $a \in A_{[X]^2}$, $\xi \in \ell^2(G_X,H)$ and $g\in G_X$. 
For each $x \in X$, 
we define $u_x \in B(\ell^2(G_X,H))$ by 
\begin{align*}
(u_x \xi)(g) 
&= V_{g;x}\,\xi(g\cup\{x\}) \in H \\
(u_x \xi)(g\cup\{x\}) 
&= V_{x;g} \,\xi(g) \in H 
\end{align*}
for $\xi \in \ell^2(G_X,H)$ and $g\in G_X$ with $x\notin g$. 

\begin{lemma} \label{L.u_xsau}
For each $x \in X$, 
$u_x$ is a self-adjoint unitary 
such that $\Ad u_x$ and $\alpha_x$ 
agree on $A_{[X]^2} \subseteq B(\ell^2(G_X,H))$. 
\end{lemma}

\begin{proof} 
For $g\in G_X$ such that $x\notin g$ 
the subspace $\ell^2(\{g, g\cup\{x\}\},H) \subseteq \ell^2(G_X,H)$ 
is invariant for $u_x$, and $u_x$ is represented on it as 
\[
u_x = 
\left(\begin{matrix} 0 & V_{g;x} \\ 
V_{x;g} & 0 \end{matrix}\right).
\]
This shows that $u_x$ is a self-adjoint unitary. 
To show that $\Ad u_x$ and $\alpha_x$ 
agree on $A_{[X]^2} \subseteq B(\ell^2(G_X,H))$, 
it suffices to see 
\begin{align*}
\Ad (V_{x;g}) \circ \alpha_g &=\alpha_{g\cup\{x\}}\circ\alpha_x \\ 
\Ad (V_{g;x}) \circ\alpha_{g\cup \{x\}} &= \alpha_g\circ \alpha_x
\end{align*}
which is Lemma~\ref{L.Vgx}. 
\end{proof}

By Lemma~\ref{L.u_xsau} we see that 
for $\{x,y\} \in [X]^2$ and $z\in X$
we have $\Ad u_x \rs_{A_{\{x,y\}}} = \Ad v_{x,y}$,
and $\Ad u_z\rs_{A_{\{x,y\}}}=\id$ if $z\notin \{x,y\}$.
In particular, 
$u_z$ commutes with $v_{x,y}$ unless $y=z$. 

\begin{lemma}\label{L.u_xv_xy}
For $\{x,y\} \in [X]^2$ 
the two self-adjoint unitaries $u_xv_{x,y}$ and $u_yv_{y,x}$ 
commute. 
\end{lemma}

\begin{proof}
Take $\{x,y\} \in [X]^2$. 
First note that for $h\in G_X$ 
we have $\alpha_h (v_{x,y})=v_{y,x} v_{x,y} v_{y,x}$ 
if $y\in h$ and $x\notin h$, and 
$\alpha_h (v_{x,y})=v_{x,y}$ otherwise.

Fix $g\in G_X$ such that $x\notin g$ and $y\notin g$. 
The subspace 
 \[
 H_g=\ell^2\big(\{g, g\cup\{x\}, g\cup\{y\},g\cup\{x,y\}\},H\big)
 \subseteq \ell^2(G_X,H)
 \]
is invariant for each of  $u_x$, $u_y$, $v_{x,y}$ and $v_{y,x}$. 
Using the observation in the beginning of this proof, 
we see that $u_x$, $u_y$, $v_{x,y}$ and $v_{y,x}$ 
are represented on $H_g$ by 
\begin{align*}
u_x&= 
\left (\begin{matrix}
0 & \! V_{g;x} \! \! & 0 & 0 \\
V_{x;g} \! & 0 & 0 & 0 \\
0 & 0 & 0 & \! \! V_{g\cup\{y\};x} \\
0 & 0 & \! \! V_{x;g\cup\{y\}} \! & 0 
\end{matrix}
\right),&
v_{x,y}&=\left(
\begin{matrix}
v_{x,y} \! & 0 & 0 & 0 \\
0 & \! v_{x,y} \! & 0 & 0 \\
0 & 0 & \!\! v_{y,x} v_{x,y} v_{y,x} \!\! & 0 \\
0 & 0 & 0 & \! v_{x,y}
\end{matrix}
\right),\\
u_y &= 
\left(
\begin{matrix}
0 & 0 & \! \! V_{g;y} \! & 0 \\
0 & 0 & 0 & \! V_{g\cup\{x\}; y}\\
V_{y;g} \! & 0 & 0 & 0 \\
0 & \! V_{y;g\cup\{x\}} \! \! & 0 & 0 
\end{matrix}
\right),&
v_{y,x}&= 
\left(
\begin{matrix}
v_{y,x} \! & 0 & 0 & 0 \\
0 & \!\! v_{x,y} v_{y,x} v_{x,y} \!\! & 0 & 0 \\
0 & 0 & \! v_{y,x} \! & 0 \\
0 & 0 & 0 & \! v_{y,x}
\end{matrix}
\right).
\end{align*}
Using the computations such as $V_{x;g\cup\{y\}}=V_{x;g}v_{x,y}v_{y,x}$, 
we see that $u_x v_{x,y}$ and $u_y v_{y,x}$ 
are represented on $H_g$ by 
\begin{align*}
u_x v_{x,y} &=\left (\begin{matrix}
0 & \!\!\! V_{g;x}v_{x,y} \!\!\! & 0 & 0 \\
V_{x;g}v_{x,y} \!\!\! & 0 & 0 & 0 \\
0 & 0 & 0 & \!\!\! V_{g;x}v_{y,x} \\
0 & 0 & \!\!\! V_{x;g}v_{y,x} \!\!\! & 0 
\end{matrix}
\right), \\
u_yv_{y,x} &= \left(
\begin{matrix}
0 & 0 & \!\!\! V_{g;y}v_{y,x} \!\!\! & 0 \\
0 & 0 & 0 & \!\!\! V_{g; y}v_{x,y}\\
V_{y;g}v_{y,x} \!\!\! & 0 & 0 & 0 \\
0 & \!\!\! V_{y;g}v_{x,y} \!\!\! & 0 & 0 
\end{matrix}
\right).
\end{align*}
The unitaries $V_{g;x},V_{x;g},V_{g;y},V_{y;g},v_{x,y}$ and $v_{y,x}$
occurring in entries of these two matrices 
commute with each others except that 
$v_{x,y}$ does not commute with $v_{y,x}$. 
Using this fact, 
one can show that both $(u_x v_{x,y})(u_y v_{y,x})$ 
and $(u_y v_{y,x})(u_x v_{x,y})$ are equal to 
\[
\left(
\begin{matrix}
0 & 0 & 0 & \!\!\! V_{g;x}V_{g;y} \\
0 & 0 & \!\!\! V_{x;g}V_{g;y}v_{x,y} v_{y,x} \!\!\! & 0 \\
0 & \!\!\! V_{g;x}V_{y;g}v_{y,x} v_{x,y} \!\!\! & 0 & 0 \\
V_{x;g}V_{y;g} \!\!\! & 0 & 0 & 0 
\end{matrix}
\right).
\]
Therefore $u_x v_{x,y}$ and $u_y v_{y,x}$ commute. 
\end{proof}

Let 
\[
B_{[X]^2} := C^*(A_{[X]^2}\cup\{u_x\}_{x \in X}) 
\subseteq B \big( \ell^2(G_X,H) \big).
\]
For a subset $Y \subseteq X$, we define
\[
B_{[Y]^2} := C^*(A_{[Y]^{2}}\cup\{u_x\}_{x \in Y}) 
\subseteq B_{[X]^2}.
\]

\begin{remark}\label{rem:cocycle_cross_pro}
The C*-algebra $B_{[X]^2}$ does not depend on 
the choices of embeddings $A_{[X]^2} \subseteq B(H)$, 
and is isomorphic to a cocycle crossed product 
$A_{[X]^2} \rtimes_{(\alpha,c)} G_X$ 
for an appropriate cocycle action $(\alpha,c)$ 
(see \cite{PaRae:Twisted} for definitions of cocycle actions 
and cocycle crossed products). 
In fact, the proof of Proposition~\ref{Prop:B=limCAR} 
shows that any C*-algebra generated by 
$A_{[X]^2}\cup\{u_x\}_{x \in X}$ with 
the relations in Lemma~\ref{L.u_xsau} and Lemma~\ref{L.u_xv_xy} 
is isomorphic to $B_{[X]^2}$. 
\end{remark}

\begin{prop}\label{Prop:B=limCAR}
The C*-algebra $B_{[X]^2}$ is 
a unital LM algebra with $\chi(B_{[X]^2}) = |X|$. 
\end{prop}

\begin{proof}
By Lemma~\ref{L.tensor.density}, 
we have $\chi(A_{[X]^2}) = |X|$. 
This implies $\chi(B_{[X]^2}) = |X|$.

We are going to show that 
$B_{[X]^2}$ is a direct limit of CAR algebras.
This implies that $B_{[X]^2}$ is LM.
For a finite subset $F \subseteq X$
and an injective map $\iota \colon F \to X\setminus F$,
define a subalgebra
$D_{(F,\iota)} \subseteq B_{[X]^2}$ by
\[
D_{(F,\iota)} :=
C^*\big(B_{[F]^2} \cup \{w_{x,\iota(x)}\}_{x \in F}\big) \subseteq B_{[X]^2}.
\]
The family $\{D_{(F,\iota)}\}_{(F,\iota)}$
of subalgebras is directed because $X$ is infinite, 
and its union is dense in $B_{[X]^2}$.
Thus it suffices to show that $D_{(F,\iota)}$ is the CAR algebra
for every finite subset $F \subseteq X$ and
every injective map $\iota \colon F \to X\setminus F$.

Take a finite subset $F \subseteq X$ and
an injective map $\iota \colon F \to X\setminus F$.
For $x \in F$,
we define
\[
u'_{x} := u_{x} \prod_{y \in F\setminus \{x\}} v_{x,y} \in D_{(F,\iota)}.
\]
which is a self-adjoint unitary. 
Since Lemma~\ref{L.u_xsau} shows 
\[
\Ad u_x \rs_{A_{[F]^2}} = \alpha_x \rs_{A_{[F]^2}} 
= \Ad \Big( \prod_{y \in F\setminus \{x\}} v_{x,y} \Big) \rs_{A_{[F]^2}}, 
\]
$u'_{x}$ commutes with the subalgebra $A_{[F]^2}$.
The family $\{u'_{x}\}_{x\in F}$ mutually commutes 
by Lemma~\ref{L.u_xv_xy}. 
For each $x \in F$,
the self-adjoint unitary $w_{x,\iota(x)} \in D_{(F,\iota)}$
commutes with $A_{[F]^2}$ 
and $\{w_{y,\iota(y)},u'_{y}\}_{y \in F\setminus \{x\}}$,
and satisfies $u'_{x}w_{x,\iota(x)} = - w_{x,\iota(x)}u'_{x}$.
Therefore $C^*(u'_{x},w_{x,\iota(x)})$ is isomorphic to $M_2(\bbC)$
for $x \in F$ by Lemma~\ref{Lem:M2},
and the family
\[
\{C^*(u'_{x},w_{x,\iota(x)})\}_{x \in F} \cup \{A_{[F]^2}\}
\]
mutually commutes.
Since $D_{(F,\iota)}$ is generated by
these mutually commuting subalgebras,
we get
\[
D_{(F,\iota)} =
A_{[F]^2} \otimes \bigotimes_{x \in F}C^*(u'_{x},w_{x,\iota(x)})
\cong \bigotimes_{|[F]^2|\times \aleph_0 +|F|}M_2(\bbC).
\]
We are done.
\end{proof}

\begin{lemma}\label{Lem:a+bv}
Let $Y$ be a nonempty proper subset of $X$.
Take $x \in Y$ and $y \in X \setminus Y$.
Then every element in $B_{[Y]^2} \subseteq B_{[X]^2}$
can be written as $a v_{x,y} + a'$
for $a, a' \in Z_{B_{[X]^2}}(A_{\{x,y\}})$.
\end{lemma}

\begin{proof}
Since $v_{x,y}$ is a self-adjoint unitary in $A_{\{x,y\}}$,
the set of all elements in the form
$a v_{x,y} + a'$ for $a, a' \in Z_{B_{[X]^2}}(A_{\{x,y\}})$
is a subalgebra of $B_{[X]^2}$.
Hence it suffices to show that
the generators $A_{[Y]^{2}}\cup\{u_z\}_{z \in Y}$ of $B_{[Y]^2}$
are in this form.
We have $A_{[Y]^{2}} \subseteq Z_{B_{[X]^2}}(A_{\{x,y\}})$
since $y \notin Y$.
We have $u_z \in Z_{B_{[X]^2}}(A_{\{x,y\}})$
for $z \in Y\setminus \{x\}$.
Finally, we get $u_x = (u_xv_{x,y})v_{x,y}$
and $u_xv_{x,y} \in Z_{B_{[X]^2}}(A_{\{x,y\}})$.
We are done.
\end{proof}

\begin{prop} \label{P:notAF}
If $|X| > \aleph_1$ then $B_{[X]^2}$ is not AF.
\end{prop}

\begin{proof}
For the sake of obtaining a contradiction, assume that $B_{[X]^2}$ is AF. 
Then by Lemma~\ref{L.LM-criterion} 
there exists a family $\{b_x\}_{x \in X}$ in $B_{[X]^2}$
with $\|u_x - b_x\| < 1$ for all $x \in X$ such that $C^*(\{b_x\}_{x \in F})
\subseteq B_{[X]^2}$ is finite-dimensional for all finite subsets $F$ of $X$.

For each $x \in X$,
there exists a countable subset $Y_x$ of $X$ with $x \in Y_x$
such that $b_x \in B_{[Y_x]^2}$.
Since $|X| > \aleph_1$, 
we can apply Lemma~\ref{Lem:aleph2} to get $\{x,y\} \in [X]^2$
such that $x \notin Y_y$ and $y \notin Y_x$.
By Lemma~\ref{Lem:a+bv},
there exists $a_x, a_x', a_y, a_y' \in Z_{B_{[X]^2}}(A_{\{x,y\}})$
such that $b_x = a_x v_{x,y} + a_x'$
and $b_y = a_y v_{y,x} + a_y'$.
Since $\|u_x - b_x\| < 1$,
we have
\[
\big\|\big((u_x - b_x)-w_{x,y}(u_x - b_x)w_{x,y}\big)/2\big\| < 1.
\]
We have
\[
\big(b_x -w_{x,y}b_xw_{x,y}\big)/2
=\big((a_x v_{x,y} + a_x')-(-a_x v_{x,y} + a_x')\big)/2
=a_x v_{x,y},
\]
and similarly $(u_x - w_{x,y}u_xw_{x,y})/2 = u_x$.
Hence we get $\|u_x - a_x v_{x,y}\| <1$.
Thus $\|u_xv_{x,y} - a_x \| <1$.
Since $u_xv_{x,y}$ is a unitary,
$a_x$ is an invertible element.
Similarly, one can show that $a_y$ is also invertible.

By the assumption,
$C^*(\{b_x, b_y\})$ is finite-dimensional.
Therefore $\{(b_xb_y)^n\}_{n=0}^\infty$ is linearly dependent.
Hence there exist $N \in \bbN$ and
$\lambda_0, \lambda_1, \ldots, \lambda_N \in \bbC$
with $\lambda_N \neq 0$
such that $\sum_{n=0}^N \lambda_n(b_xb_y)^n=0$.
We can write
\[
\sum_{n=0}^N \lambda_n(b_xb_y)^n
= \sum_{n=0}^N \lambda_n\big((a_x v_{x,y} + a_x' )(a_y v_{y,x} + a_y')\big)^n
= \sum_{v \in V} f_v v
\]
where
\[
V
:= \{1, v_{x,y}, v_{y,x}, v_{x,y}v_{y,x}, v_{y,x}v_{x,y},
v_{x,y}v_{y,x}v_{x,y}, v_{y,x}v_{x,y}v_{y,x}, \ldots, (v_{x,y}v_{y,x})^N\}
\]
and for each $v \in V$,
$f_v \in Z_{B_{[X]^2}}(A_{\{x,y\}})$ is
a sum of products of
$\lambda_0, \lambda_1, \ldots, \lambda_N \in \bbC$
and $a_x, a_x', a_y, a_y' \in Z_{B_{[X]^2}}(A_{\{x,y\}})$.
Since $V \subseteq A_{\{x,y\}}$ is linearly independent,
we get $f_v =0$ for all $v \in V$ by Lemma~\ref{Lem:lin.indep}.
In particular,
$f_{(v_{x,y}v_{y,x})^N} = \lambda_N(a_xa_y)^N \in Z_{B_{[X]^2}}(A_{\{x,y\}})$ 
is $0$.
This cannot  happen because $\lambda_N \neq 0$
and both $a_x$ and $a_y$ are invertible.
Thus we get a contradiction.
We are done.
\end{proof}

\begin{remark}
When $|X|= \aleph_0$, $B_{[X]^2}$ is a UHF algebra 
(in fact CAR algebra) by Glimm's theorem \cite[Theorem~1.13]{Glimm:On}. 
When $|X|= \aleph_1$, $B_{[X]^2}$ is a unital AM algebra 
by Proposition~\ref{Prop:B=limCAR} and Theorem~\ref{Thm1}~(1).
In this case one can show that $B_{[X]^2}$ is not UHF 
in a similar (but much more complicated) way 
to the proof of Proposition~\ref{P.B-X}~(2) (see \cite{FaKa:NonseparableII}). 
\end{remark}

\begin{remark}
As  we pointed out in  Remark~\ref{rem:crosspro}, 
the examples in Section~\ref{Sec:AMnotUHF} 
of unital AM algebras which are not UHF 
are obtained as crossed products 
of UHF algebras by the group $\bbZ /2\bbZ$. 
The examples in this section of 
unital LM algebras which are not AM 
are obtained as {\em cocycle} crossed products 
(see Remark~\ref{rem:cocycle_cross_pro}). 
However we do not know the following. 
\end{remark}

\begin{problem}
Find an example of a unital LM algebra which is not AM 
such that it is obtained as a crossed product 
of a unital AM (or UHF) algebra by a discrete group. 
\end{problem}

\begin{remark}
We can solve the non-unital version of this problem 
using the examples in this section. 
In fact, 
by \cite[Corollary~3.7]{PaRae:Twisted} 
the tensor product $B_{[X]^2} \otimes K$ 
is obtained as an (ordinary) crossed product 
of $A_{[X]^2} \otimes K$ by the group $G_X$ 
where $K := K\big(\ell^2(G_X)\big)$ 
is the non-unital AM algebra of all compact operators 
on the Hilbert space $\ell^2(G_X)$. 
Thus for every cardinal $\kappa > \aleph_1$,
there exists an example of 
a {\em non-unital} LM algebra with character density $\kappa$
which is not AM 
such that it is obtained as a crossed product 
of a non-unital AM algebra by a discrete group. 
Note that $B_{[X]^2} \otimes K$ is not AM 
if $B_{[X]^2}$ is not AM
because every corner of an AM algebra is AM.

The same comments can be applied to LF and AF 
instead of LM and AM. 
\end{remark}

\section{Representation density and character density}
\label{S.rd}

The purpose of this section is to give an answer to 
the half of the question raised by Masamichi Takesaki
when the second author gave a talk on this paper. 
We could not answer the other half (Problem~\ref{P.Takesaki}). 
The proof uses the construction (Proposition~\ref{P.chi_neq_rd}) 
that was given by Bruce Blackadar 
when the first author gave a talk. 
Both authors would like to thank 
Masamichi Takesaki and Bruce Blackadar. 

For a Hilbert space $H$, 
we also denote by $\chi(H)$ 
the smallest cardinality of a dense subset of $H$.
Note that for an infinite-dimensional Hilbert space $H$ and 
an infinite set $X$, 
we get $\chi(H) = |X|$ 
if and only if $H$ is isomorphic to $\ell^2(X)$. 

\begin{definition} \label{Def:rd} 
The \emph{representation density} $\rd(A)$ of a C*-algebra $A$ 
is the smallest cardinal $\chi(H)$ 
where $H$ is a Hilbert space on which 
$A$ can be faithfully represented. 
\end{definition}

Note that both the representation density $\rd$ 
and the character density $\chi$ (Definition~\ref{D.cd}) 
are monotonic in the sense that if $A$ is a
subalgebra of $B$ then the density of $B$ is not smaller than the
density of $A$.

Since these cardinal invariants of C*-algebras were apparently not
considered previously, 
the reader will hopefully excuse us for
starting this section by listing a few trivial statements.

\begin{lemma} \label{L.rd.1} 
For every C*-algebra $A$ we have that
\begin{align*}
\chi(A)&\geq \sup\big\{|X|: \text{$X$ is a family of commuting
projections in $A$}\big\}\\
\rd(A)&\geq\sup\big\{|X|: \text{$X$ is a family of nonzero orthogonal
projections in $A$}\big\}.
\end{align*}
\end{lemma}

\begin{proof} 
For the first part note that if $p$ and $q$ are
distinct commuting projections then $\|p-q\|=1$. 
The second part is obvious.
\end{proof}

\begin{lemma} \label{L.rd.2} 
For every infinite-dimensional Hilbert space $H$ we have
\[
\chi\big(\cB(H)\big)
=|\cB(H)|
=2^{\chi(H)}.
\]
\end{lemma}

\begin{proof} 
Let us choose an infinite set $X$ with $|X|=\chi(H)$, 
and identify $H$ with $\ell^2(X)$. 
For a subset $Y \subseteq X$, 
let $p_Y \in \cB(H)$ be 
the projection onto the subspace $\ell^2(Y) \subseteq H$. 
Then $\{p_Y\}_{Y \subseteq X}$ is 
a family of commuting projections of size $2^{|X|}$. 
Thus we have $\chi(\cB(H)) \geq 2^{|X|}$ by Lemma~\ref{L.rd.1}. 
For $x,y \in X$, $p_{\{x\}}\cB(H)p_{\{y\}}$ is one dimensional, 
and the map 
\[
\cB(H) \ni T \mapsto \big(p_{\{x\}}Tp_{\{y\}}\big)_{x,y\in X} 
\in \prod_{x,y\in X} \big(p_{\{x\}}\cB(H)p_{\{y\}}\big)
\cong \prod_{x,y\in X} \bbC
\]
is injective. 
Hence we get 
$\chi(\cB(H)) \leq |\cB(H)|\leq |\bbC|^{|X\times X|} =2^{|X|}$.
We are done. 
\end{proof}

If $K=2^{2^{\aleph_0}}$ with the product topology 
then $C(K) \cong \bigotimes_{2^{\aleph_0}}\bbC^2$ is an
abelian C*-algebra with character density $2^{\aleph_0}$ and
representation density $\aleph_0$. The first claim follows by
Lemma~\ref{L.tensor.density}. 
The second claim follows from the fact that $K$ is, 
being a product of $2^{\aleph_0}$ separable spaces, 
separable by the Hewitt--Marczewski--Pondiczery Theorem 
(see e.g., \cite[Corollary 2.3.16]{En:Topology}).
See also Corollary~\ref{C.rd.1}, Theorem~\ref{T.rd.1} and
Problem~\ref{P.Takesaki}.

\begin{lemma} 
For every C*-algebra $A$ we have
$\rd(A)\leq \chi(A)\leq 2^{\rd(A)}$.
\end{lemma}

\begin{proof} 
Choose a subset $X \subseteq A$ with $|X| = \chi(A)$. 
For each $x \in X$, 
there exists a cyclic representation 
$\pi_x\colon A \to \cB(H_x)$ with $\|\pi_x(x)\|=\|x\|$ 
(see \cite[Corollary~II.6.4.9]{Black:Operator}). 
Since $H_x$ has a cyclic vector for $\pi_x$, 
we have $\chi(H_x) \leq \chi(A)$. 
Then the representation
\[
\pi := \bigoplus_{x \in X} \pi_x 
\colon A \to \cB\Big( \bigoplus_{x \in X}H_x \Big)
\]
is faithful, and 
\[
\chi\Big( \bigoplus_{x \in X}H_x \Big)
= \sum_{x \in X} \chi (H_x)
\leq |X| \times \chi(A) = \chi(A)
\]
Hence $\rd(A) \leq \chi(A)$. 
The second inequality $\chi(A)\leq 2^{\rd(A)}$ follows from
Lemma~\ref{L.rd.2}. 
\end{proof}

\begin{lemma}\label{L.rd.tp}
Let $X_0 \ni x \mapsto \xi_x \in H$ be a map 
from a set $X_0$ to a Hilbert space $H$ 
such that $|X_0| > \chi(H)$. 
Then for every $\e > 0$, 
there exists $X_1 \subseteq X_0$ with $|X_1| > \chi(H)$ 
such that $\|\xi_{x}-\xi_{y}\|<\e$ 
for every $x,y \in X_1$. 
\end{lemma}

\begin{proof}
Choose a dense subset $Y \subseteq H$ with $|Y|=\chi(H)$.
For each $x \in X_0$ there exists $\eta(x) \in Y$ 
such that $\|\xi_x - \eta(x)\|<\e/2$. 
Since $|X_0| > \chi(H) = |Y|$, 
there exists $\eta \in Y$ such that the set 
$X_1 := \{x \in X_0 : \eta(x)=\eta \} \subseteq X_0$ 
satisfies $|X_1| > \chi(H)$. 
Then for every $x,y \in X_1$, we get 
\[
\|\xi_{x}-\xi_{y}\| 
\leq \|\xi_x - \eta\| + \|\xi_y - \eta\|<\e.
\qedhere
\]
\end{proof}

\begin{prop} \label{P.rd.1} 
For a family $\{A_x\}_{x \in X}$ of 
nonabelian unital C*-algebras, 
the representation density of 
the tensor product $A=\bigotimes_{x \in X} A_x$ 
is at least $|X|$.
\end{prop}

\begin{proof} 
Assume the contrary and fix a faithful representation
$\pi\colon A \to \cB(H)$ for a Hilbert space $H$ 
with $|X| > \chi(H)$. 
Note that this assumption implies that $X$ is uncountable. 
For each $x \in X$, 
fix $a_x$ and $b_x$ in the unit ball of $A_x$ 
such that $a_x b_x \neq b_x a_x$.
Since $\pi$ is faithful, 
we can choose a vector $\xi_x \in H$ such that
\[
\pi(a_x b_x - b_x a_x)\xi_x \neq 0.
\]
Since $X$ is uncountable, 
there exist $\delta>0$ and a subset $X_0 \subseteq X$ 
with $|X_0| > \chi(H)$ such that for all $x \in X_0$ 
we have 
\[
\|\pi(a_x b_x - b_x a_x)\xi_x\| \geq \delta.
\]
Set $\e = \delta/4>0$. 
In this proof, 
we write $a \approx_{\e} b$ if $\| a - b \| < \e$. 
Since $|X_0| > \chi(H)$, 
we can apply Lemma~\ref{L.rd.tp} to 
$\{\xi_x\}_{x \in X_0}$ and $\e>0$ 
to get $X_1 \subseteq X_0$ with $|X_1| > \chi(H)$ 
such that $\xi_{x} \approx_{\e} \xi_{y}$ 
for every $x,y \in X_1$. 
By applying Lemma~\ref{L.rd.tp} three more times 
to $\{ \pi(a_x)\xi_x \}_{x \in X_1}$ and so on, 
we get $X_4 \subseteq X_1$ with $|X_4| > \chi(H)$ 
such that 
\[
\pi(a_x)\xi_{x} \approx_{\e} \pi(a_y)\xi_{y}, \quad 
\pi(b_x)\xi_{x} \approx_{\e} \pi(b_y)\xi_{y}, 
\]
\[
\pi(b_xa_x)\xi_{x} \approx_{\e} \pi(b_ya_y)\xi_{y}
\]
for every $x,y \in X_4$. 
Since $|X_4| > \chi(H) \geq \aleph_0$, 
we can take two distinct $x,y \in X_4$. 
Then we have 
\begin{align*}
\pi(a_xb_x)\xi_{x} = 
\pi(a_x) \pi(b_x)\xi_{x} \approx_{\e} 
\pi(a_x) \pi(b_y)\xi_{y} =
\pi(a_x b_y)\xi_{y} 
\approx_{\e} 
\pi(a_x&b_y)\xi_{x} \\
&\rotatebox{90}{=} \\
\pi(b_xa_x)\xi_{x} \approx_{\e} 
\pi(b_ya_y)\xi_{y} =
\pi(b_y) \pi(a_y)\xi_{y}\,\approx_{\e}
\pi(b_y) \pi(a_x)\xi_{x} = 
\pi(b_y&a_x)\xi_{x} 
\end{align*}
because $a_x \in A_x \subseteq A$ and $b_y \in A_y \subseteq A$ 
commute. 
Thus we get 
\[
\big\|\pi(a_xb_x - b_xa_x)\xi_{x}\big\| < 4\e = \delta,
\]
which is a contradiction.
This completes the proof.
\end{proof}

\begin{corollary} \label{C.rd.1} 
If $A$ is a UHF algebra then $\chi(A)=\rd(A)$. \qed
\end{corollary}

With the possible exception of the algebras $A_{X,Y}$ as defined in
\S\ref{Sec:LMnotUHF}, each example of an AM, or even LM, algebra
given so far has a UHF subalgebra with the same character density.
Since the algebras $A_{X,Y}$ are tensor products of separable
algebras, Proposition~\ref{P.rd.1} implies that for each AM or LM
algebra $A$ so far defined in this paper we have $\chi(A)=\rd(A)$.
We are going to show that $\chi(A)$ 
can be any cardinality between $\rd(A)$ and $2^{\rd(A)}$
for unital AM algebras $A$. 

Let $X$ be an infinite set. 
As in Section~\ref{Sec:AMnotUHF}, 
let $A_x$ be a C*-algebra
generated by two self-adjoint unitaries $v_x, w_x$
with $v_x w_x = - w_x v_x$ for each $x \in X$, 
and let $A_X := \bigotimes_{x \in X} A_x$. 
By Lemma~\ref{Lem:M2}, 
$A_x \cong M_2(\bbC)$ for each $x \in X$ 
and hence $A_X \cong \bigotimes_X M_2(\bbC)$ is a UHF algebra.
For each $Y\subseteq X$,
we set 
\[
A_Y := \bigotimes_{x \in Y} A_x \subseteq A_X.
\]

We are going to use the GNS representation of $A_X$ 
associated with the unique tracial state of $A_X$. 
For the reader's convenience we explain what it is. 
For each finite subset $F \subseteq X$, 
there exists a unique linear functional $\tau_F \colon A_F \to \bbC$ 
satisfying the trace condition $\tau_F(ab)=\tau_F(ba)$ 
for $a,b \in A_F$ and the normalized condition $\tau_F(1)=1$.
If $|F|=n$, then we have $\tau_F = 2^{-n}\text{Tr}$ 
where $\text{Tr}$ is the usual trace of $A_F \cong M_{2^{n}}(\bbC)$. 
It is easy to see that $\tau_F$ is positive and faithful, 
that is, $\tau_F(a^*a) > 0$ for all $a \in A_F \setminus \{0\}$. 
Let $A_X^{\text{fin}} := \bigcup_{F\subseteq X} A_F \subseteq A_X$ 
where $F$ runs all finite subsets of $X$. 
By the uniqueness of the tracial state $\tau_F$, 
we get $\tau_{F'}\rs_{A_F}=\tau_F$ 
for two finite subsets $F \subseteq F' \subseteq X$.
Thus we get a linear map $\tau \colon A_X^{\text{fin}}\to \bbC$ 
such that $\tau \rs_{A_F} = \tau_F$ 
for every finite subset $F \subseteq X$. 
Although we do not need it, 
we would like to remark that 
$\tau$ can be extended to the unique tracial state of $A_X$ 
(cf.\ \cite[Lemma~I.9.5]{Dav:C*}). 
We define an inner product on $A_X^{\text{fin}}$ 
by $A_X^{\text{fin}}\times A_X^{\text{fin}} 
\ni (a,b) \mapsto \tau(ab^*) \in \bbC$. 
Then the completion $H_X$ of $A_X^{\text{fin}}$ 
with respect to the norm coming from the inner product defined as above 
becomes a Hilbert space. 
The embedding from $A_X^{\text{fin}}$ to $H_X$ is denoted by 
$A_X^{\text{fin}} \ni a \mapsto \hat{a} \in H_X$. 
The image of this embedding is dense in $H_X$. 
For each finite subset $F \subseteq X$ and each $a \in A_F$, 
it is easy to see that 
the map $\hat{b} \mapsto \widehat{ab}$ 
extends to a bounded operator on $H_X$. 
Thus we get a $*$-ho\-mo\-mor\-phism $\pi_F \colon A_F \to B(H_X)$ 
such that $\pi_F(a)(\widehat{b})=\widehat{ab}$ 
for $a \in A_F$ and $b \in A_X^{\text{fin}}$. 
We have $\pi_{F'}\rs_{A_F}=\pi_F$ 
for two finite subsets $F \subseteq F' \subseteq X$.
Since the family $\{\pi_{\{x\}}[A_{\{x\}}]\}_{x\in X}$ 
mutually commutes, 
we get a representation $\pi\colon A_X \to B(H_X)$ 
such that $\pi \rs_{A_F} = \pi_F$ 
for every finite subset $F \subseteq X$. 
This representation is called the GNS representation 
associated with $\tau$. 
Since $\pi(a)(\widehat{a^*})=\widehat{aa^*}\neq 0$ 
for all $F \subseteq X$ and all $a \in A_F \setminus\{0\}$, 
$\pi$ is injective. 
In order to simplify the notation 
we identify $A_X$ with the subalgebra $\pi[A_X]$ of $B(H_X)$. 

\begin{lemma}\label{L.chi(HX)}
We have $\chi(H_X) = |X|$. 
\end{lemma}

\begin{proof}
Since the union of finite-dimensional subspaces 
$\{\hat{a} \in H_X \mid a \in A_F\}$ for finite subsets $F \subseteq X$ 
is dense in $H_X$, 
we have $\chi(H_X) \leq |X|$. 
For distinct $x,y \in X$, 
we have $\tau(u_xu_y)=0$ because 
\begin{align*}
\tau(u_xu_y)
&=\tau(w_x(w_xu_xu_y))
=\tau((w_xu_xu_y)w_x)\\
&=\tau(w_xu_x(w_xu_y))
=\tau(w_x(-w_xu_x)u_y)
=-\tau(u_xu_y).
\end{align*}
Hence we get 
\[
\|\widehat{u_x}-\widehat{u_y}\|^2 = \tau((u_x-u_y)(u_x-u_y))
= \tau( 2- 2u_xu_y) =2 
\]
for all $x,y \in X$ with $x \neq y$. 
This shows that $\chi(H_X) \geq |X|$. 
Thus we get $\chi(H_X) = |X|$. 
\end{proof}

We can consider the power-set $\cP(X)$ of a set $X$ 
as an abelian group with respect to the symmetric difference. 
This group is naturally isomorphic to the direct product of $X$ 
copies of $\bbZ/2\bbZ$.
For $g\in \cP(X)$ consider an
automorphism of $A_X$ defined by
\[
\alpha_g =\bigotimes_{x\in g} \Ad v_x.
\]
Then $\alpha$ defines an action of $\cP(X)$ on $A_X$.
For each $g\in \cP(X)$, 
the automorphism $\alpha_g$ 
preserves the subalgebra $A_F \subseteq A_X$ 
and satisfies $\tau_F \circ \alpha_g = \tau_F$ 
for every finite subset $F \subseteq X$. 
Hence we get an element $u_g \in B(H_X)$ 
such that $u_g(\widehat{b})=\alpha_g\widehat{(}b)$ 
for $b \in A_X^{\text{fin}}$. 

\begin{lemma} 
The elements $\{u_g\}_{g \in \cP(X)} \subseteq B(H_X)$ 
are self-adjoint unitaries 
satisfying $u_g a u_g = \alpha_g(a)$ and $u_gu_h=u_{gh}$ 
for $a \in A_X \subseteq B(H_X)$ and $g,h\in \cP(X)$. 
\end{lemma}

\begin{proof}
Take $g \in \cP(X)$. 
Since $\alpha_g$ preserves $\tau$, 
the element $u_g^* \in B(H_X)$ satisfies 
$u_g^*(\widehat{b})=\alpha_g^{-1}\widehat{(}b)$ 
for $b \in A_X^{\text{fin}}$.
Hence $u_g$ is a unitary. 
This is self-adjoint because $\alpha_g^{-1}=\alpha_g$. 
The latter two equalities 
follow from the equations $\alpha_g(a\alpha_g(b))=\alpha_g(a)b$ 
and $\alpha_g(\alpha_h(b))=\alpha_{gh}(b)$ for $b \in A_X$. 
\end{proof}

\begin{definition}\label{D.group}
For an infinite set $X$ and a subgroup $\Gamma \subseteq \cP(X)$ 
we define 
\[
B_{X,\Gamma} := C^*(A_X\cup\{u_g\}_{g \in \Gamma}) 
\subseteq B(H_X).
\]
\end{definition}

\begin{remark}
One can show that $B_{X,\Gamma}$ is isomorphic 
to the crossed product $A_X \rtimes_{\alpha} \Gamma$. 
In particular $B_X$ in Section~\ref{Sec:AMnotUHF} 
is isomorphic to $B_{X,\Gamma}$ 
for $\Gamma = \{\emptyset, X\} \cong \bbZ /2 \bbZ$. 
\end{remark}

\begin{prop} \label{P.chi_neq_rd}
The C*-algebra $B_{X,\Gamma}$ satisfies 
$\chi(B_{X,\Gamma})=|X|+|\Gamma|$
and $\rd(B_{X,\Gamma})=|X|$. 
\end{prop}

\begin{proof} 
We have $\chi(A_X)=|X|$ by Lemma~\ref{L.tensor.density}. 
On the other hand, 
we have $\chi(C^*(\{u_g\}_{g \in \Gamma})) \geq |\Gamma|$ 
by Lemma~\ref{L.rd.1}
because $\{(u_{g}+1)/2\}_{g \in \Gamma}$ 
is a family of commuting projections.
Since $B_{X,\Gamma}$ is generated by $A_X$ and $\{u_g\}_{g \in \Gamma}$, 
we get 
\[
|X|+|\Gamma| = \max\{|X|, |\Gamma|\} 
\leq \chi(B_{X,\Gamma}) \leq |X|+|\Gamma|
\]
This shows $\chi(B_{X,\Gamma}) = |X|+|\Gamma|$.
Since $B_{X,\Gamma} \subseteq B(H_X)$, 
we have $\rd(B_{X,\Gamma}) \leq \chi(H_X) = |X|$ 
by Lemma~\ref{L.chi(HX)}. 
We also have 
$\rd(B_{X,\Gamma}) \geq \rd(A_X) = |X|$ by Corollary~\ref{C.rd.1}. 
Hence we get $\rd(B_{X,\Gamma})=|X|$.
\end{proof}

\begin{prop} \label{P.chi_neq_rd2}
The unital C*-algebra $B_{X,\Gamma}$ is AM
if every finite subset of $\Gamma$ is included 
in a subgroup generated by $g_1,g_2,\dots, g_n \in \Gamma$ 
which are infinite and mutually disjoint. 
\end{prop}

\begin{proof} 
Take mutually disjoint infinite elements $g_1,g_2,\dots, g_n \in \Gamma$. 
Take a finite subset $F$ of $X$ and 
choose $x_i\in g_i\setminus F$ for $i=1,2, \ldots, n$. 
Let $\Lambda$ be the set of 
all such data $\lambda = (\{g_i\}_{i=1}^n,F,\{x_i\}_{i=1}^n)$, 
and define 
\[
D_\lambda := 
C^*\big(\{u_{g_i}\}_{i=1}^n \cup A_F \cup \{w_{x_i}\}_{i=1}^n\big)
\subseteq B_{X,\Gamma}.
\]
By the assumption of $\Gamma$, 
the family $\{D_\lambda\}_{\lambda \in \Lambda}$ of subalgebras 
is directed and its union is dense in $B_{[X]^2}$.
We are going to show $D_\lambda \cong M_{2^{m+n}}(\bbC)$ 
for $\lambda = (\{g_i\}_{i=1}^n,F,\{x_i\}_{i=1}^n)$ as above where $m=|F|$. 
This implies that $B_{[X]^2}$ is AM, 
and hence completes the proof. 
For $i \in \{1,2,\dots, n\}$ 
define 
\[
u_i'=u_{g_i}\prod_{x\in F\cap g_{i}} v_x \in D_\lambda.
\]
Since 
\[
\Ad u_{g_i} \rs_{A_F} 
= \Ad \Big( \prod_{x\in F\cap g_{i}} v_x \Big) \rs_{A_F}, 
\]
$u'_i$ is a self-adjoint unitary 
and commutes with the subalgebra $A_{F}$.
It is easy to see that 
the family $\{u'_{i}\}_{i=1}^n$ mutually commutes.
Since $x_i\in g_i\setminus F$ 
and $g_i$ is disjoint from $g_j$ for $j\neq i$, 
we have that $w_{x_i}$ commutes with $A_F$ 
and $\{u'_{j},w_{x_j}\}_{j\neq i}$.
Finally $w_{x_i}$ and $u'_i$ anti-commute 
because so do $w_{x_i}$ and $u_{g_i}$. 
Therefore $C^*(u'_{i},w_{x_i})$ is isomorphic to $M_2(\bbC)$
for $i \in \{1,2,\dots, n\}$ by Lemma~\ref{Lem:M2},
and the family
\[
\{C^*(u'_{i},w_{x_i})\}_{i=1}^n \cup \{A_F\}
\]
mutually commutes.
Since $D_{\lambda}$ is generated by
these mutually commuting subalgebras,
we get
\[
D_{\lambda} =
\Big( \bigotimes_{i=1}^{n} C^*(u'_{i},w_{x_i}) \Big) \otimes A_{F}
\cong \bigotimes_{n+|F|}M_2(\bbC) \cong M_{2^{n+m}}(\bbC),
\]
as required.
\end{proof}

\begin{remark}
For finite $g \in \cP(X)$, 
we have $\alpha_g = \Ad \big( \prod_{x\in g} v_x \big)$. 
From this fact, 
one can show that $B_{X,\Gamma}$ is not AM 
if $\Gamma$ contains a finite nonempty element $g$ 
(one can also show that $B_{X,\Gamma}$ is always AF). 
Thus in order for $B_{X,\Gamma}$ to be AM it is necessary 
that every $g \in \Gamma\setminus\{\emptyset\}$ is infinite. 
One can show that this is also sufficient 
although its proof becomes significantly complicated 
compared with Proposition~\ref{P.chi_neq_rd2}. 
We shall not need such generality 
for proving Theorem~\ref{T.rd.1}. 
\end{remark}

\begin{remark}
One can show that $B_{X,\Gamma}$ is not UHF 
when $|X| \geq \aleph_1$ and $\Gamma \neq \{\emptyset\}$ 
in a similar way to the proof of Proposition~\ref{P.B-X}~(2). 
We omit the proof because we do not need this 
(see the proof of Theorem~\ref{T.rd.1} for some special cases). 
One can also show that $Z_{B_{X,\Gamma}}(A_X)=\bbC 1$ holds 
when every $g \in \Gamma\setminus\{\emptyset\}$ is infinite 
(even in the case $\chi(A_X) < \chi(B_{X,\Gamma})$). 
This shows that a generalization of 
question \cite[Problem~8.3]{Dix:Some} 
for nonseparable AM algebras has a very strong negative 
answer (see Corollary~\ref{Cor:DixProb8.3}). 
The authors would like to thank Bruce Blackadar 
for pointing out the phenomenon $Z_{B_{X,\Gamma}}(A_X)=\bbC 1$. 
This strong phenomenon does not occur for UHF algebras
because we can show $\chi(Z_B(A)) = \chi(B)$ for 
a subalgebra $A$ of a UHF algebra $B$ with $\chi(A) < \chi(B)$, 
and hence in this case $Z_B(A)$ is huge. 
\end{remark}

\begin{lemma}\label{L:Boolean}
For every cardinal $\kappa$ 
with $|X| \leq \kappa \leq 2^{|X|}$, 
there exists a subgroup $\Gamma \subseteq \cP(X)$ with $|\Gamma|=\kappa$ 
such that every finite subset of $\Gamma$ is included 
in a subgroup generated by $g_1,g_2,\dots, g_n \in \Gamma$ 
which are infinite and mutually disjoint. 
\end{lemma}

\begin{proof}
Take a subset $Y \subseteq \cP(X)$ with $|Y| = \kappa$. 
Let $\Gamma_0$ be 
the Boolean subalgebra of $\cP(X)$ generated by $Y$, 
that is the smallest subset of $\cP(X)$ containing $Y$ and 
closed under taking unions, intersections and complements. 
Then $\Gamma_0$ is a subgroup of $\cP(X)$ 
with $|\Gamma_0|=\kappa$. 
Choose a bijection $\iota \colon X\times \bbN \to X$ 
and define an injective homomorphism 
\[
\varphi \colon \cP(X) \ni g \mapsto \iota[g\times \bbN] \in \cP(X).
\]
Let $\Gamma := \varphi[\Gamma_0] \subseteq \cP(X)$. 
Then every finite subset of $\Gamma$ is included 
in a finite Boolean subalgebra of $\Gamma$. 
If $g_1, g_2,\ldots, g_n \in \Gamma$ 
are the atoms of this subalgebra 
then they clearly satisfy the requirements. 
\end{proof}

\begin{theorem} \label{T.rd.1} 
For every pair of infinite cardinals $\kappa$ and $\nu$ 
with $\kappa \geq \aleph_1$ and $\nu \leq \kappa \leq 2^\nu$, 
there exists a unital AM algebra 
of character density $\kappa$ 
and representation density $\nu$ 
which is not UHF. 
\end{theorem}

\begin{proof}
For $\kappa = \nu \geq \aleph_1$, 
the example $B_X$ in Proposition~\ref{P.B-X}
for $|X|=\kappa$ 
is a unital AM algebra 
of character density $\kappa$ 
and representation density $\nu$ 
which is not UHF.
Suppose $\nu < \kappa \leq 2^\nu$. 
Take a set $X$ with $|X| = \nu$. 
By Lemma~\ref{L:Boolean}, 
there exists a subgroup $\Gamma \subseteq \cP(X)$ with $|\Gamma|=\kappa$ 
satisfying the assumption of Proposition~\ref{P.chi_neq_rd2}.
Then $B_{X,\Gamma}$ is a unital AM algebra 
of character density $\kappa$ 
and representation density $\nu$ 
by Proposition~\ref{P.chi_neq_rd} and Proposition~\ref{P.chi_neq_rd2}. 
This is not UHF by Corollary~\ref{C.rd.1}. 
\end{proof}

From Theorem~\ref{T.rd.1} we have the following.

\begin{corollary} \label{Cor:AM separable} 
There is a unital AM algebra 
faithfully represented on a separable Hilbert space
that is not a UHF algebra. \qed
\end{corollary}

This corollary answers 
a half of the question raised by Masamichi Takesaki. 
The following is the other half which we could not answer. 

\begin{problem}\label{P.Takesaki}
Is there an LM algebra faithfully
represented on a separable Hilbert space
which is not AM?
\end{problem}

Since $\chi\big(B(\ell^2(\bbN))\big)=2^{\aleph_0}$, 
by Theorem~\ref{Thm1} (1) there is no such a C*-algebra 
if we assume the continuum hypothesis
$2^{\aleph_0}=\aleph_1$. 
We do not know what happens if we do not
assume the continuum hypothesis.

\subsection*{Acknowledgment}
Many of the results presented in this paper were proved while both
authors were visiting the Fields Institute in Fall 2007 and in early
2008. I.F. would like to thank colleagues from his department, in
particular Juris Stepr\=ans and Man Wah Wong, for making his stay at
the Fields Institute possible. Both authors would like to thank
George Elliott, Toshihiko Masuda, Narutaka Ozawa, 
N. Christopher Phillips, Juris Stepr\=ans, Reiji Tomatsu and  
Andrew Toms  for illuminating conversations. We would also like to thank the
anonymous referee for several useful remarks. 

\providecommand{\bysame}{\leavevmode\hbox to3em{\hrulefill}\thinspace}
\providecommand{\MR}{\relax\ifhmode\unskip\space\fi MR }
\providecommand{\MRhref}[2]{%
  \href{http://www.ams.org/mathscinet-getitem?mr=#1}{#2}
}
\providecommand{\href}[2]{#2}

\end{document}